\documentclass[12pt]{amsart} 

\usepackage{amssymb}
\usepackage{amscd}
\usepackage{xcolor} 
\usepackage{graphicx} 
\usepackage{enumerate}

\usepackage{amsmath}
\usepackage{amssymb}
\usepackage{amsrefs}

\newcommand\cee{\mathfrak{c}}

\theoremstyle{plain}
\newtheorem{theorem}{Theorem}[section]

\newtheorem{proposition}[theorem]{Proposition}
\newtheorem{corollary}[theorem]{Corollary}
\newtheorem{definition}[theorem]{Definition}

\theoremstyle{definition}

\newtheorem{question}{Question}

\newtheorem{claim}{Claim}

\theoremstyle{remark}

\numberwithin{equation}{section}

\DeclareMathOperator{\Fn}{\operatorname{Fn}}
\DeclareMathOperator{\val}{\mathop{val}}
\DeclareMathOperator{\dom}{\mathop{dom}}

\title{On the cardinality of
separable pseudoradial spaces
}

\author[A. Dow]{Alan Dow}
\address{Department of Mathematics,
University of North Carolina at Charlotte, 
Charlotte, NC 28223}
\email{adow@uncc.edu}

\author[I. Juh\'asz]{Istv\'an Juh\'asz}
\address{Alfr\'ed R\'enyi Institute of Mathematics, Hungarian Academy of
Sciences}\email{
 juhasz@renyi.hu}

\thanks{First author supported by  NSF grant DMS-1501506 and\\  Second author
supported by NKFIH grant number  K 129211}

\date{\today}

\keywords{
 tightness, pseudoradial}
 
\subjclass{ 54A35 }

\begin{document}
\begin{abstract}
The aim of this paper is to consider questions concerning
the possible maximum cardinality of various {\em separable
  pseudoradial} (in short: SP) spaces. 
The most intriguing question here 
is if there is in ZFC a regular (or just Hausdorff) SP space of
cardinality $> \mathfrak{c}$. While this question is left open, we establish
a number of non-trivial results that we list below.

\begin{itemize}

\item It is consistent with $MA\,+\, \mathfrak{c} = \aleph_2$ that
  there is a {\em countably tight} and compact 
SP space of cardinality $2^{\mathfrak{c}}$.

\item If $\kappa$ is a measurable cardinal
 then in the forcing
extension obtained by adding $\kappa$ many Cohen reals, every {\em
  countably tight} 
regular SP space has
cardinality at most $\mathfrak{c}$.

\item If $\kappa > \aleph_1$ Cohen reals are added to a model of GCH then in the extension every 
pseudocompact SP space with a countable dense set of isolated points has
cardinality at most $\mathfrak{c}$.

\item If $\mathfrak{c} \le \aleph_2$ then there is a 0-dimensional SP space 
with a countable dense set of isolated points that has
cardinality greater than $\mathfrak{c}$.
\end{itemize}
\end{abstract}
\maketitle

%\tableofcontents
\bibliographystyle{plain}

\section{Introduction}
The class of pseudoradial spaces is a natural and well-studied
class that is a generalization of the radial property in
the same way that the class of sequential spaces is a generalization
of the class of Frechet-Urysohn spaces.  For a cardinal $\kappa$,
a sequence $\{ x_\alpha : \alpha < \kappa\}$ is said to converge
 to $x$ in a space $X$, if for every neighborhood $U$ of $x$ there is a
  $\beta <\kappa$ such that $\{ x_\alpha : \beta < \alpha < \kappa\}$ is
  a subset of $U$. A set $A$ is radially closed in $X$ if, for every
  cardinal $\kappa\leq |X|$,   no $\kappa$-sequence of points of $A$
  converges to a point not in $A$.  
  A space $X$ is pseudoradial if every radially closed set is closed.
  A set $A$ is sequentially closed if no $\omega$-sequence of points
  from $A$ converges to a point not in $A$ and a space is
  sequential if every sequentially closed set is closed.

  As stated in the abstract, we are exploring the question of
whether separable 
pseudoradial spaces of cardinality greater than $\mathfrak c$ exist.
We may say that a space is large if it has cardinality
greater than $\mathfrak c$. We thank A. Bella for informing
us that it was shown in
\cite{ArhBella93} that it is consistent
that the usual product space $2^{\omega_1}$ 
is SP.

One of the most interesting results about compact pseudoradial spaces
arose when Sapirovskii \cite{SapPseudo}
proved that the continuum hypothesis implied
that a compact space is pseudoradial so long as it is sequentially
compact. A space is sequentially compact if every infinite sequence
has a limit point.
This was improved in \cite{JSpseudo} where it was shown
that it follows from $\mathfrak c\leq\aleph_2$ that compact
sequentially compact spaces are pseudoradial and that this bound
can not be improved in ZFC. In this paper we are able to show that
this
same assumption of $\mathfrak c\leq\aleph_2$ is sufficient to produce
examples of regular Hausdorff 
separable pseudoradial spaces of cardinality equal
to $2^{\mathfrak c}$. In this paper we will restrict our investigation
to regular
Hausdorff spaces and note that $2^{\mathfrak c}$ is the upper bound on
the cardinality of any separable space.

Returning to the class of compact separable spaces, we recall
that every pseudoradial compact space  is sequentially compact.
More generally  in
 a pseudoradial space, a countable discrete set is closed
if it contains no converging sequence, and so clearly,
a compact pseudoradial space is sequentially compact. It is well-known
that the sequential closure of a countable set has cardinality at
most $\mathfrak c$ and therefore every sequential separable space
has cardinality at most $\mathfrak c$.  Balogh \cite{Balogh}
proved that the proper forcing axiom implies that every compact
space of countable tightness is sequential, and therefore
compact separable spaces of countable tightness have cardinality at
most $\mathfrak c$. Of course this paper of Balogh's was in
answer to the celebrated Moore-Mrowka problem.
Recall that the proper forcing axiom also
implies that $\mathfrak c = \aleph_2$ \cites{Boban,Justin}.  
This background motivates one
to ask more about separable pseudoradial spaces of countable tightness
both with and without the extra assumption of compactness.

\bigskip
   
\section{Martin's Axiom}
This project began when we were made aware of a question
in connection to the Moore-Mrowka problem
posed 
by S. Spadaro in MathOverFlow.  We thank K.P. Hart for
bringing the question to our attention. We refer
the reader to \cite{Kunen} for details about Martin's Axiom
and to \cite{Balogh} for the statement of the proper
forcing axiom.

%\url{https://mathoverflow.net/questions/347835/}
%\url{a-variant-of-the-moore-mrowka-problem}

\begin{verse}
  
Assume  $MA_{\aleph_1}$. Is it true that every compact pseudoradial
space of countable tightness is sequential? 

\end{verse}

We answer this question in the negative but will rely on
quoting two related results from the literature. It would be too
ambitious to reproduce the proofs from these publications. However
we can connect the investigations there to the current one.

\begin{definition} A space $X$ is initially $\aleph_1$-compact
  if every open cover of cardinality at most $\aleph_1$
  has  a finite subcover.
\end{definition}

An initially $\aleph_1$-compact space is countably compact
and a first-countable initially $\aleph_1$-compact space
is sequentially compact.
It is shown in \cite{DowSide}*{7.1} that every compactification of
an initially $\aleph_1$-compact space of countable tightness
has countable tightness. It was proven
in \cite{JKL} that 
there is  a model of $\mathfrak c=\aleph_2$ 
in which there exists a separable
first-countable locally compact
initially
$\aleph_1$-compact space that is not compact.
The one-point compactification of this space
is compact, separable, sequentially compact,
countably tight and not sequential. Since
$\mathfrak c\leq\aleph_2$ holds in this model,
this space is also pseudoradial. Finally, it
was shown explicitly in \cite{DowFeng}*{5.11}
(and implicitly in \cite{DowSide}*{6.3}) that
one can perform a further ccc forcing to make
$\mbox{MA}_{\aleph_1}$ and $\mathfrak c=\aleph_2$
 hold while ensuring
that the original space generates a
space in the extension with all the same properties.
This answers Spadaro's question in the negative.

Now we consider our generalization of Spadaro's question
and ask if Martin's Axiom is sufficient to ensure
that compact separable pseudoradial spaces of countable
tightness have cardinality at most $\mathfrak c$. For
this we use the example constructed in
\cite{DowSide}*{5.5,6.3} that had, in answer to a question of
Arhangelskii, an example of an initially $\aleph_1$-compact
first-countable space of cardinality greater than $\mathfrak c$.
We cite the exact statement that we will need.
This was also shown to hold in a model of Martin's Axiom (MA)
and $\mathfrak c=\aleph_2$. Clearly this space was not
compact and it also contained a separable closed subset
that had a compactification of cardinality $2^{\mathfrak c}$.

\begin{proposition}\cite{DowSide}*{5.5,6.3} There is a model of MA
 $+~\mathfrak c   =\aleph_2$ in which there is a separable
first-countable\label{propBig}
initially   $\aleph_1$-compact space $X$
with the property
that $\beta X\setminus X$ has
cardinality $2^{\aleph_2}$ and contains no infinite compact subsets.
\end{proposition}

\begin{corollary} It is consistent with $\mbox{MA}_{\aleph_1}$
  that there is a compact separable pseudoradial space of countable
  tightness and cardinality greater than $\mathfrak c$.
\end{corollary}

\begin{proof} Let $X$ be the space and in the model as
  stated in Proposition \ref{propBig}. We show that
  $\beta X$ is the desired example. The only property that needs
  to be established is that $\beta X$ is pseudoradial and has
  countable tightness. We remarked above that every compactification
  of a initially $\aleph_1$-compact space has countable tightness.
  We are working in a model of $\mathfrak c =\aleph_2$, so to
  prove that $\beta X$ is pseudoradial, it is enough to
  prove that it is sequentially compact. The space $X$ is
  sequentially compact so we consider an infinite countable subset $D$
  of $\beta X\setminus X$. The closure of $D$ is not contained
  in $\beta X\setminus X$ by the statement in Proposition
  \ref{propBig}, hence there is a point $x\in X$
  that is a limit point of $D$. It 
  is easily checked that if $\{ U_n : n\in\omega\}$ is a local
  base in $X$ for $x$ consisting of cozero subsets of $X$,
 then the family $\{ \beta X\setminus \mbox{cl}_{\beta X}(X\setminus U_n) :
 n\in\omega \}$ is a local base for $x$ in $\beta X$.
 It follows then that there is a sequence from $D$ converging
 to $x$, and this completes the proof that $\beta X$
 is sequentially compact and pseudoradial.
  \end{proof}

\section{cardinality of separable pseudoradial spaces}

In this section we begin our investigation of
separable pseudoradial spaces in the absence of
the assumption of compactness. We first consider
the effect of the proper forcing axiom in the context
of a strengthening of countable tightness and then,
using large cardinals, we establish the consistency
of there being no large separable pseudoradial spaces
with countable tightness.

We introduce a natural generalization of
the property of a set being 
 sequentially closed 
in a space $X$. In particular, 
in this next definition, we would say 
that  a sequentially closed set is ${<}\aleph_1$-closed.

\begin{definition} For a cardinal $\kappa$ and a set
$A$ in a space $X$, say that $A$ is ${<}\kappa$-radially closed
if for every cardinal $\lambda<\kappa$, no $\lambda$-sequence
of points of $A$ converges to a point of $X$ outside of $A$.

For a set $A$ in a space $X$, let $A^{(<\kappa)}$ denote the smallest
${<}\kappa$-radially closed set containing $A$. 
\end{definition}

Needless to say, $A^{(<\kappa)}$ exists in every space since
the intersection of ${<}\kappa$-radially closed sets is ${<}\kappa$-radially closed.
Alternatively, $A^{(<\kappa)}$ can be recursively constructed in  a manner
analogous to the usual constructions of the sequential closure. 

\begin{proposition} Let $A$ be a subset\label{chain} of a space $X$ and let $\kappa$
be a cardinal. By induction on $\alpha \leq \kappa^{+}$, define
 $A_\alpha\subset X$ as follows:
 \begin{enumerate}
 \item $A_0 = A$,
 \item for limit $\alpha\leq\kappa^+$, $A_\alpha = \bigcup \{A_\beta : \beta < \alpha\}$,
 \item $A_{\alpha+1}$ is the set of all points of $X$ such that there is a 
 cardinal $\lambda<\kappa$
 and a $\lambda$-sequence of points of $A_\alpha$ that converges to $x$.
 \end{enumerate}
 Then $A_{\kappa^+} = A^{(<\kappa)}$.
 \end{proposition}

 \begin{proof} It follows easily, by induction on $\alpha\leq \kappa^+$
 that $A_\alpha$ is a subset of $A^{(<\kappa)}$. So it suffices
 to note that $A_{\kappa^+}$ is itself ${<}\kappa$-radially closed.
 \end{proof}

A $\kappa$-sequence $\{x_\alpha : \alpha < \kappa\}$ is a free sequence
 \cite{Arhangelskii}
 in a space $X$
 if, for all $\delta<\kappa$, the initial segment $\{x_\alpha : \alpha <\delta\}$
 and the final segment $\{ x_\alpha : \delta\leq \alpha <\kappa\}$ have disjoint closures. 
 The tightness degree, $t(X)$, of a space $X$ is of course a well-known cardinal
 invariant. Similarly, the invariant, $F(X)$, is the supremum of the lengths of
 free sequences of $X$. Of course Arhangelskii showed that $t(X)\leq F(X)$ holds
 for any compact space $X$, and it was shown by Bella \cite{Bella86} that this inequality also
 holds for pseudoradial spaces. We record this for future reference.

\begin{proposition}[\cite{Bella86}]
   A pseudoradial space of uncountable tightness\label{freeseq} contains uncountable free
   sequences. 
   \end{proposition}

\long\def\eat#1{\relax}
\eat{   \begin{proof}
   Let $A$ be a subset of $X$. 
   Every point of $A^{(<\omega_1)}$ is in the closure of a countable
   subset of $A$. Let $\kappa$ be the minimal cardinal so
   that there is a point $x$ in $A^{(<\kappa)}$ that is not in the closure
   of a countable subset of $A$.  
   Let $\{ A_\alpha : \alpha \leq \kappa^+\}$ be the sequence of sets
   as defined in Proposition \ref{chain}.   Choose $\alpha <\kappa^+$
   minimal so that 
   there is a point $x\in A_{\alpha+1}$  that is not in the closure
   of a countable subset of $A$. Fix a cardinal $\lambda < \kappa$ so
   that there is a $\lambda$-sequence $\{ x_\xi : \xi<\lambda\}\subset
    A_\alpha$ that converges to $x$. We may choose $\lambda$ to
    be minimal.  It follows that $x$ is not in the closure of any 
    countable subset of $\{ x_\xi : \xi< \lambda\}$.   
  We recursively choose an increasing sequence $\{ \xi_\gamma :\gamma\in \omega_1\}
    \subset \lambda$ satisfying that 
    the closure of $\{ x_{\xi_\gamma} : \gamma <\delta\}$ is disjoint from
    the closure of $\{ x_\xi : \xi_{\delta} \leq \xi<\lambda\}$. 
    Assume that $\delta<\omega_1$ and that we have chosen $\{ \xi_\gamma : \gamma<\delta\}
    \subset \lambda$. Choose a neighborhood $U_\delta$ of $x$ satisfying that
    the closure of $U_\delta$ is disjoint from the closure of $\{ x_{\xi_\gamma} : \gamma < \delta\}$.
     Choose $\xi_\delta<\lambda$ large enough so that $\{ x_\xi : \xi_\delta \leq \xi  <\lambda\}$
     is contained in $U_\delta$. 
   \end{proof}
   }

\subsection{ PFA and countable tightness}

\begin{proposition} PFA implies that every separable
regular
pseudoradial space of cardinality greater than $\cee$
contains uncountable free sequences.
\end{proposition}

\begin{proof}  
Let $X$ be a separable regular pseudoradial space of cardinality
greater than $\cee$. Since we simply wish to prove that $X$
contains an uncountable free sequence, we may assume, by
  Corollary \ref{freeseq},  that $X$ has countable tightness.
  Let $\kappa$ be a large enough regular cardinal
so that  $X$ is an element of $H(\kappa)$ (see \cite{Kunen}*{IV \S6}).
Let $X$ be an element of an elementary submodel $M$
of $H(\kappa)$ where $M$ is chosen so that $|M|=
 2^{\aleph_1}=\cee$ and so that 
every subset of $M$ of cardinality at most $\aleph_1$
is an element of $M$. We note that $Y=X\cap M$ is 
 ${<}\aleph_2$-radially closed. We may assume
 that $\omega$ is a dense subset of $X$.
 Let $\mathcal W_z = \{ W\subset \omega : 
 z\in \mbox{int}_X\,\mbox{cl}_X(W)\}$.  
Define the family $\mathcal H_z$ to be all countable
subsets of $Y$ that have $z$ in their closure. Let
 ${\mathcal F}_z$ denote the family of all closed
 subsets of $Y$ that contain a member of the family
  $\{ \mbox{cl}_Y (H) : H\in \mathcal H_z\}$. 

\begin{claim} 
$\mathcal F_z$ is a countably complete maximal filter of closed subsets of $Y$.
\end{claim}

\bgroup

\def\proofname{Proof of Claim:}

\begin{proof}  Let $\{ H_n : n \in\omega\}$ be a subset of $\mathcal H_z$. 
Let $F_\omega = \bigcap\{ \mbox{cl}_Y)(H_n) : n\in \omega\}$ and assume
there is $W\in \mathcal W_z$ such that $\mbox{cl}_X(W)$ is disjoint
from $F_\omega$. Since $W\in M$, we can replace each $H_n$ 
by $\mbox{int}_X\,\mbox{cl}_X(W)\cap H_n$ since the latter is also
in $\mathcal H_z$. However, this is impossible since $z\in\mbox{cl}_X(H_n)$
for all $n\in\omega$,  by elementarity,
 there are points of $Y$ in $F_\omega$.  Now it follows that $z$ is in the
 closure of $F_\omega$, and by countable tightness, there is an $H_\omega\in
 \mathcal H_z$ with $H_\omega\subset F_\omega$. 
\end{proof}

\egroup

It is evident from its definition that $\mathcal F_z$ has a base of separable sets.
Only minor modifications of any number of published proofs in connection to
the Moore-Mrowka problem (e.g. \cite{PrahaIII}*{Theorem 3.3})
 that these hypotheses on the filter $\mathcal F_z$
are sufficient to conclude that
 there is a proper poset $P$ that will force
an uncountable free sequence in the space $X$. We provide  a sketch.
Fix a neighborhood assignment $\mathcal U_Y=\{  U_y : y\in Y\}$ so that for each
$y\in Y$, $z$ is not in the closure of $U_y$. 
A member $p$ of the poset
 $P$ is a function into $Y$ with a finite domain $\mathcal M_p$.
  The domain $\mathcal M_p$ is a finite $\in$-chain of countable elementary
  submodels of $H(\kappa)$ satisfying that $\{X,Y,\mathcal F_z, \mathcal U_Y\}$
  are members of each. For each $\{M_1,M_2\}\subset \mathcal M_p$ 
  with $M_1\in M_2$, the value of $p(M_1)$ is an element of $M_2\cap F$
  for all $F\in \mathcal F_z\cap M_1$. The poset is ordered by the conditions
  that $p<q$ providing $p\supset q$ and for all
    $M_1\in \mathcal M_p\setminus \mathcal M_q$, if there is a (minimal)
     $M_2\in \mathcal M_q$ such that $M_1\in M_2$, then 
       $p(M_1)\in U_{p(M_3)}$ for all $M_3\in \mathcal M_q$ satisfying
       that $p(M_2)\in U_{p(M_3)}$. We omit the proof that $P$ is proper
       (see any of [ ]). It is easily shown that for each $\delta\in\omega_1$,
       the set $D_\delta = \{ p\in P  :  (\exists M_p\in \mathcal M_p)~~ \delta\in M_p\}$
       is  a dense subset of $P$. By PFA, there is a filter $G\subset P$ such that
$G\cap D_\delta$ is not empty for all $\delta<\omega_1$. This implies there
is an uncountable set $C_G\subset \omega_1$ such that 
$\gamma\in C_G$ if and only if there is an $M_\gamma\in \bigcup\{ \mathcal M_p  :
p\in G\}$ with $M_\gamma \cap \omega_1=\gamma$. 
For each $\gamma\in C_G$, choose any $p_\gamma\in G$ such that $M_\gamma\in 
\mathcal M_{p_\gamma}$.  Since $G$ is a filter, we may assume that 
for all $\gamma\in C_G$, $M_{\min(C_G)}\in \mathcal M_{p_\gamma}$.

 For each $\gamma\in C_G$, also let
 $F_\gamma  = \bigcap  \mathcal F_z\cap M_\gamma$.  
The sequence
 $\{ x_\gamma = p_\gamma(M_\gamma) : \gamma\in C_G\}$ 
 satisfies that  for each $\delta\in C_G$,  $F_\delta$ contains the set
  $\{ x_\gamma : \delta \leq \gamma\}$. It also follows that, for $\gamma<\delta$
  both in $C_G$, the closure of
   $U_{x_\gamma}$ is disjoint from $F_\delta$ because of the condition
   that $z$ is not in the closure of $U_{x_\gamma}$. Finally, 
 for each $\min(C_G)<\delta\in C_G$, there is a $\beta\in C_G\cap \delta$ such
 that $\{x_\xi : \beta < \xi\in C_G\cap \delta\}$ is contained $U_{x_\delta}$. 
Indeed, then there is a finite $S_\delta\subset\delta{+}1$ such that 
$\{ x_\xi : \xi\in C_G\cap \delta\}$ is contained in $\bigcup\{ U_{x_\xi} : \xi\in S_\delta\}$.
This completes the proof that the closure of $\{ x_\xi : \xi\in C_G\cap \delta\}$ is
disjoint from the closure of $\{ x_\xi : \delta < \xi \in C_G\}$. 
\end{proof}

In compact spaces, the property of being countably tight is equivalent
to the property that every subset has countable $\pi$-character
\cite{Boris1}. A family
of non-empty open sets $\mathcal U$ is a local $\pi$-base in $X$ at $x$ if every
neighborhood of $x$ contains a member of $\mathcal U$. A point has
countable $\pi$-character if it has a countable local $\pi$-base. It was discovered
by Eisworth \cite{Todd} that in the absence of compactness, the property
of hereditary countable $\pi$-character was an important strengthening
of countable tightness. In particular, it was shown in \cite{DowTodd}
and \cite{picharacter} that it was consistent that countably compact subsets
of such spaces are closed. Using this result we note the following
application.

\begin{corollary} PFA implies that every separable countably
compact pseudoradial regular space
has cardinality at most $\cee$
if the space also has
 hereditary countable $\pi$-character.
\end{corollary}

\begin{proof} Assume that $X$ is a separable countably compact pseudoradial
regular space. The sequential closure of the countable dense set is 
countably compact and has cardinality $\cee$. It was proven in
\cite{picharacter} that PFA implies that 
countably compact subsets are closed in such spaces.
\end{proof}

\begin{question} Let $X$ be a countably compact regular pseudoradial
space of countable tightness. Does PFA imply any of the following?
\begin{enumerate} 
\item $X$  has cardinality at most $\cee$.
\item Countably compact subsets of $X$ are closed.
\item Every subset of $X$ has countable $\pi$-character.
\end{enumerate}
\end{question}

\subsection{Consistency of no large countably tight examples}

In this section we prove that it is consistent, modulo large cardinals, 
that every separable regular pseudoradial space of countable tightness
has cardinality at most $\cee$. We give the result using a measurable
cardinal for simplicity, but  the same proof can be modified to work
for a weakly compact cardinal.   A cardinal $\kappa$ is measurable
if there is a $\kappa$-complete free ultrafilter on $\kappa$. 

Before proving the theorem we make note of a simple property shared
by pseudoradial
spaces.

\begin{proposition} Assume\label{limits}
 that a point $z$ of a pseudoradial 
regular space $X$ 
is a limit point  of a countable subset $A$ of $X$.  Let $A^{(1)}$ 
denote the set of points of $X$ that are limit points of a converging
sequence from $A$. Then $z$ is in the closure of $A^{(1)}\setminus A$.
\end{proposition}

\begin{proof} Since $z$ is a limit point of $A$
we may assume that $z\notin A$  and, to be clear,
 we note that $A\subset A^{(1)}$.
   If $z$ is an element of $A^{(1)}$ then
 there is nothing
to prove. Let $W$ be a neighborhood of $z$ and assume that
 $A_W = \overline{W}\cap A^{(1)}$ is contained in $A$. 
 Certainly $z$ is in the closure of $A_W$ but this contradicts
 that $X$ is pseudoradial 
 since $A_W$ is clearly radially closed. 
 \end{proof}

\begin{theorem} If $\kappa$ is a measurable cardinal\label{mble}
 then in the forcing
extension obtained by adding $\kappa$ many Cohen reals, every 
separable regular pseudoradial space of countable tightness has
cardinality at most $\cee$.
\end{theorem}

\begin{proof} Let $\kappa$ be a measurable cardinal. By \cite{Jech}*{(8.7),10.20} ,
 there is a normal $\kappa$-complete ultrafilter $\mathcal U$
 on $\kappa$.  This means that for every $U\in \mathcal U$ and
 regressive function $f$ on $U$ (i.e. $f(\alpha)<\alpha)$ for
 all $\alpha\in U$)
 there is a   $U_1\in \mathcal U$ satisfying that $ f(\lambda)=f(\mu)$
 for all $\lambda,\mu\in U_1$. Note also that it follows that each 
  $U\in \mathcal U$ is a stationary subset of $\kappa$ \cite{Jech}*{10.19}. 
  More surprisingly,
 there is an element $U$ of $\mathcal U$ consisting only of regular
 limit cardinals. To review this, first we note that the function sending $\mu^+$ to $\mu$
 is regressive on the set of successor cardinals and so this set can not
 be an element of $\mathcal U$. Now we prove that there is a member
 of $\mathcal U$ consisting only of regular cardinals.
This is because if
the function $f(\lambda)=\mbox{cf}(\lambda)$ was regressive
on a  set $U$ of $\mathcal U$, i.e. every member of $U$ would have
the same cofinality $\delta<\kappa$,  then we would have $\delta$-many
regressive functions (to elements of cofinal sequences)  and one of these
could not be constant on a member of $\mathcal U$. Needless to say,
 it is also true that $\kappa$ is a strongly inaccessible cardinal.
    
    \begin{claim} If $\{ S_\xi : \xi \in\kappa\}$ is any\label{pressdown}
     sequence of countable
    subsets of $\kappa$, then there is a set $\bar S$ and a set $U\in \mathcal U$
    such that for all $\lambda,\mu\in U$ with $\lambda<\mu$, 
      $\bar S=S_\lambda\cap \lambda$ and $S_\lambda\subset \mu$.
    \end{claim}
     \bgroup
     
     \def\proofname{Proof of Claim:\/}
     
     \begin{proof}
     For each $\xi\in \kappa$, let $\delta_\xi<\omega_1$ be the order-type of $S_\xi\cap \xi$
     and let $\{ \beta^\xi_\eta : \eta < \delta_\xi\}$ be an enumeration of $S_\xi\cap \xi$.
     Choose $U_0\in \mathcal U$ and $\delta<\omega_1$ so that $\delta_\lambda=\delta$
     for all $\lambda\in U_0$. Similarly, for each $\eta<\delta$, choose $\beta_\eta <\kappa$
     and $U_{\eta+1}\in\mathcal U$ 
     so that $\beta^\lambda_\eta=\beta_\eta$ for all $\lambda\in U_{\eta+1}$. 
     Let $\bar U \in \mathcal U$ equal $U_0\cap \bigcap \{ U_{\eta +1}: \eta<\delta\}$.
     Choose any cub $C\subset \kappa$ so that for all $\lambda\in C$ and $\gamma<\lambda$,
      $S_\gamma\subset \lambda$. Now it follows that $C\cap \bar U\in \mathcal U$
      has the required properties.      
     \end{proof}
     
     We say that an indexed subset $\{ p_\xi : \xi < \delta\}$ of the poset
      $\Fn(\kappa,2)$ is a $\Delta$-system if there is a $\bar p\in \Fn(\kappa,2)$
      such that for all $\xi\neq\eta < \delta$, $p_\xi\restriction \dom(\bar p) = \bar p$
      and the members of the
      family $\{ \dom(p_\xi)\setminus \dom(\bar p) : \xi<\delta\}$ are pairwise disjoint.
      We say that $\bar p$ is the root of the $\Delta$-system. We record two well-know
      facts about Cohen forcing.
     
     \begin{claim} If  $\{ p_\xi : \xi < \delta\} $ is\label{Delta}
      a $\Delta$-system in $\Fn(\kappa,2)$
     with root $\bar p$, then for every generic filter $G$ with $\bar p\in G$ and
     every infinite subset $J\subset \delta$
      the set $\{ \xi \in J : p_\xi\in G\}$ has cardinality $|J|$.
     \end{claim}
     
     \begin{claim} For\label{smallDelta}
      any uncountable regular cardinal $\mu\leq \kappa$ and
      sequence $\{ p_\xi : \xi <\mu\}$, there is a set $J$ cofinal in $\mu$ so that
       $\{ p_\xi : \xi\in J\}$ is a $\Delta$-system.
       \end{claim}

     \egroup
        
    Now suppose that $\theta>\kappa$ is a cardinal and that 
    in the extension by the poset $\Fn(\kappa,2)$
     there is a regular pseudoradial topology on $\theta$
     in which $\omega$ is dense and which has countable tightness.
     
     If $G$ is a  $\Fn(\kappa,2)$-generic filter,
     then in $V[G]$, by  \cite{Kunen}*{VII}), $\cee = \kappa$ and,
     for all infinite cardinals $\lambda<\kappa$, $2^\lambda=\kappa$. 
Similarly, the cardinals  of $V$ 
and the cardinals of $V[G]$ coincide. 
     It therefore follows that, in $V[G]$, the cardinality of 
      $\omega^{(<\kappa)}$ is at most $\kappa$.
We will use the convention that a nice name of a subset of $\theta$
is a set $\dot A$ of the form $\bigcup \{ \{\alpha\}\times A_\alpha : \alpha\in I\}$
where $I$ is any subset of $\theta$ and for each $\alpha\in I$,
 $A_\alpha$ is a countable subset of $\Fn(\kappa,2)$. 
 We will say that $\dot A$ is a nice name of a subset of $I$ to mean
 that $\{ \alpha \in \theta : A_\alpha \neq \emptyset \} \subset I$. 
 When we say that $\dot A$
  is an $I$-name we will mean in this sense.
 Usually
  each $A_\alpha$ is taken to be an antichain but it
  will be convenient to use our generalization.  
   If $G$ is a $\Fn(\kappa,2)$-generic
  filter, then $\val_{G}(\dot A)$ is defined as $\{ \alpha : G\cap A_\alpha\neq\emptyset\}$. 
  It is shown in \cite{Kunen} that, in $V[G]$, every subset of $\theta$ is
  equal to $\val_{G}(\dot A)$ for some nice name $\dot A$ (in the ground model).
Moreover,
 for every infinite subset $J$ of $\theta$ in $V[G]$ there is
  a ground model set $I$ and a nice name of a subset of $I$ for $J$,
  where $I$ and $J$ have the same cardinality in $V[G]$.
  Similarly, it   
  can be shown that for
  every subset $\mathcal A$ of the power set of $\theta$ in $V[G]$,
  there is a family $\mathcal N$ of nice names of subsets of $\theta$
  such that $\mathcal A$ is equal to 
 $\{ \val_G(\dot A) : \dot A\in \mathcal N\}$. Therefore, we may suppose
 that we have families $\{ \mathcal N_\alpha : \alpha\in \theta\}$ 
 where, for each $\alpha<\theta$, $\mathcal N_\alpha$ is the family
 of nice names of the open subsets of $\theta$ that contain $\alpha$. 
 
 \bigskip
  
 Another subset of $\theta$ that is of interest is the set that will
 evaluate 
 to $\omega^{(<\kappa)}$ in $V[G]$. There is a set $I\subset \theta$
 of cardinality at most $\kappa$ so that there is a nice name $\dot Y$
 of a subset of $I$ satisfying that, for any generic filter $G\subset
 \Fn(\kappa,2)$, $\val_G(\dot Y)$ contains
  $\omega^{(<\kappa)}$. 
  Next, we may choose a condition $\bar p\in G$
   and an ordinal $z\in \theta\setminus I$ such that $\bar p$ forces
   that there is a $\kappa$-sequence from $\omega^{(<\kappa)}$
   that converges to $z$. Again we may suppose that there is a
    sequence $\{ \dot y_\xi : \xi < \kappa\}$ where for each $\xi<\kappa$,
    $\dot y_\xi$ is a nice name of a single ordinal that is forced by $\bar p$
    to be in $\dot Y$ and so that the sequence is forced to converge
    to $z$. It follows that each $\dot y_\xi$ is an $I_\xi$-name for some
    countable (or finite) set $I_\xi$. 
    \bigskip
    
    For each $\xi < \kappa$, we may choose a pair $(\alpha_\xi,p_\xi)$ that
    is in the name $\dot y_\xi$ and such that $\bar p$
    is compatible with $ p_\xi$.   
     That is, if $\bar p\cup p_\xi\in G$, then $\val_{G}(\dot y_\xi)$
    is equal to $\alpha_\xi$ and is in $\omega^{(<\kappa)}$.
     A minor modification of Claim \ref{pressdown}
    show that there is a $\bar q \leq \bar p$ and a set $U_0\in \mathcal U$
    such that, for all $\lambda<\mu$ with both from $U_0$, 
    $\bar q = (\bar p\cup p_\lambda)\restriction\lambda$
    and $\dom(p_\lambda)\subset \mu$. 
     
    \bigskip
    
    We define a new nice name $\dot Z = \{ (\alpha_\lambda,\bar p\cup p_\lambda) : \lambda\in U_0\}$
    and observe that, by Claim \ref{Delta}, if $\bar p\in G$ (for a $\Fn(\kappa,2)$-generic filter $G$)
    the sequence $\val_G(\dot Z)$ is a $\kappa$-sequence that converges to $z$.
    Let $C_0$ be a cub in $\kappa$ satisfying that, for each $\mu\in C_0$,
     $U_0\cap \mu$ is cofinal in $\mu$. Since $z$ is forced to not be an element of $\omega^{(<\kappa)}$,
      we may choose, for each $\mu\in C_0$, a nice name $\dot A_\mu\in \mathcal N_z$, 
      so that there is a $I_\mu\subset U_0\cap \mu$ and a nice name $\dot J_\mu$ 
      for a subset of $I_\mu$ satisfying that
      $\dot J_\mu$ is forced to be cofinal in $\mu$ and 
       $\alpha_\xi$ is not  $\dot A_\mu$ for all $\xi\in \dot J_\mu$. 
       We emphasize  that,
     in $V[G]$, $z$ is not in the closure of $\{ \alpha_\lambda : \lambda\in \dot J_\mu\}$.
       Recall
       that a condition $q\in \Fn(\kappa,2)$ forces that an ordinal $\xi\in \dot J_\mu$
       if there is some pair $(\xi,p)\in \dot J_\mu$ with $q\leq p$. 
       Similarly, $q$ forces that $\xi\in \dot Z$ if $\xi\in U_0$ and $q\leq \bar p\cup p_\xi$.
       \bigskip
       
       For each $\lambda\in U_0$ and $\lambda<\mu\in C_0$, choose a pair $(\xi(\lambda,\mu), q(\lambda,\mu))$
       so that
       \begin{enumerate}
       \item $\lambda \leq \xi(\lambda,\mu) \in U_0\cap \mu$,
       \item $q(\lambda,\mu)\in \Fn(\kappa,2)$ and $q(\lambda,\mu)$ forces $\xi(\lambda,\mu)\in \dot Z$,
         \item $q(\lambda,\mu)$ forces that $\xi(\lambda,\mu)\in \dot J_\mu$. 
       \end{enumerate}
              There is a cub $C_1$ satisfying
        that, for all $\zeta\in C_1$,
        $\{ q(\lambda,\mu) : \mu<\zeta \ \mbox{and}\ \lambda\in U_0\cap \mu\}$
        is a subset of $\Fn(\zeta,2)$. 
            Choose any $U_1\in \mathcal U$ so that $U_1\subset C_1$
       and every element of   $ U_1$ 
 is a regular cardinal.
 
 Applying Claim \ref{smallDelta}, we may choose, for each $\mu\in U_1$,
  a strictly increasing function $h_\mu: \mu \mapsto U_0\cap \mu$ so that 
     the sequence $\{ q(h_\mu(\alpha), \mu) : \alpha  \in \mu\}$ is
     a $\Delta$-system. For each $\mu\in U_1$, let $\tilde q_\mu$ denote
     the root of this $\Delta$-system,   and, by possibly shrinking $U_1$,
      we can also assume that the sequence $\{ \tilde q_\mu : \mu\in U_1\}$ 
      is a $\Delta$-system.

 Now we come to a  stronger version of Claim \ref{Delta}.
      
        \begin{claim} There is a strictly increasing function $h$ from $\kappa$
        into $U_0$  and a sequence $\{ (\bar\xi_{\alpha}, \bar q_{\alpha}) : \alpha<\kappa\}$
         such that
          for all $\omega<\delta<\kappa$, there is a  set $U_\delta\in \mathcal U$
          so that, for all $\mu\in U_\delta$,
    $$ \{(h(\alpha),\bar\xi_{\alpha}, \bar q_{\alpha}) : \alpha<\delta\} 
            = \{ (h_\mu(\alpha),\xi(h_\mu(\alpha) ,\mu), 
            q(h_\mu(\alpha),\mu) \restriction\mu)   : \alpha<\delta\}~.$$
                    \end{claim}
        
        \bgroup
        
        \def\proofname{Proof of Claim:\/}
        
  \begin{proof}
   Choose any enumeration $e:\kappa \mapsto \kappa\times \kappa\times\Fn(\kappa,2)$.
    Choose a cub
   $C_e\subset\kappa$ so that for all $\gamma < \delta\in C_e$,
    $e(\gamma)\in\delta\times \delta\times\Fn(\delta,2)$.
   For each $\alpha\in \kappa $, the function sending $\mu$ to the ordinal $\gamma_\mu$
   satisfying that  $(h_\mu(\alpha),\xi(h_\mu(\alpha)) , q(h_\mu(\alpha),\mu)\restriction\mu) = 
   e(\gamma_\mu)$ is regressive on a set in $\mathcal U$. Therefore there is a 
$  W_\alpha\in \mathcal U$ 
 and a value $g(\alpha)\in \kappa$ so that  $$(h_\mu(\alpha) ,\xi(h_\mu(\alpha),\mu),
  q(h_\mu(\alpha) ,\mu)\restriction\mu) = 
   e(g(\alpha))\ \ \mbox{ for all}\ \  \mu\in W_\alpha~.$$
   Let  
   $(h(\alpha),\bar\xi_{\alpha}, \bar q_{\alpha})$ be defined so as to equal 
    $e(g(\alpha))$.
   Now, for each $\delta<\kappa$, the set $U_\delta = \bigcap_{\alpha<\delta}W_\alpha$
   is an element of $\mathcal U$ and, after some straightforward unravelling, we have
   that for all $\mu\in U_\delta$, 
   $$ \{(h(\alpha),\bar\xi_{\alpha}, \bar q_{\alpha}) : \alpha<\delta\} 
            = \{ (h_\mu(\alpha),\xi(h_\mu(\alpha) ,\mu), 
            q(h_\mu(\alpha),\mu) \restriction\mu)   : \alpha<\delta\}~~.$$
    \end{proof}
  \egroup

 The family $\{ \bar q_\alpha : \alpha \in \kappa\}$ is a $\Delta$-system of
 $\Fn(\kappa,2)$ since for each $\delta<\kappa$, there is a $\mu\in U_1$ such
 that
   the initial segment $\{ \bar q_\alpha : \alpha < \delta\}$  is a subset
   of the $\Delta$-system $\{ 
   q(h_\mu(\alpha),\mu) \restriction\mu  : \alpha<\delta\}~~.$
   In addition, each $\bar q_\alpha$ forces that $\bar \xi_\alpha$ is an
   element of the set $\{ \dot y_\zeta : \alpha\leq \zeta <\kappa\}$. 
   
   We briefly pass to the extension $V[G]$ and note that, by Claim \ref{Delta},
   the set $J = \{ \alpha : \bar q_\alpha\in G\}$ is cofinal in $\kappa$. 
   Since the sequence $\{ \val_G(\dot y_\zeta) : \zeta < \kappa\}$ converges
   to $z$, and by countable tightness, there is a countable subset $J_1$
   of $J$ such that $\{ \bar\xi_\alpha : \alpha\in J_1\}$
  has $z$ in its closure. 
   Now we apply   Proposition \ref{limits} to the point $z$
   and the countable set $A=\{ \bar\xi_\alpha : \alpha\in J_1\}$ and
   we choose   a countable set $J_2$ of $\theta$ (i.e. the space $X$) 
   so that $J_2 \subset A^{(1)}\setminus A$ and $z$ is in the closure
   of $J_2$. 
   
   Now we return to $V$ and we choose countable
    $I_1\subset \kappa$ and countable $I_2\subset\theta $ so that,
    in $V[G]$, $J_1\subset I_1$ and $J_2\subset I_2$.  Fix a nice
    name, $\dot J_2$, for the subset $J_2$ of $I_2$ 
    and notice that the
    name $\dot J_1$ for $J_1$ is simply the set
     $\{ (\alpha , \bar q_\alpha) : \alpha \in I_1\}$.
       Next we want to choose
      sufficiently many nice names for sequences contained in
      $\{\bar\xi_\alpha : \alpha\in J_1\}$ 
      that converge to points of $J_2$. We can do this by a simple
      examination of the name $\dot J_2$. For each ordinal $x\in I_2$
      and condition $p\in \Fn(\kappa,2)$ such that $(x,p)$
      is an element of the name $\dot J_2$ (thus $p\Vdash x\in J_2$)
      choose a nice name $\dot S(x,p)$ of a subset of $I_1$ such
      that $p$ forces that $ \dot S(x,p)$ is a subset of $J_1$ 
      and that $\{ \bar\xi_\alpha : \alpha \in \dot S(x,p)\}$
      converges to $x$. Certainly
      we have only chosen countably many nice names of the form
       $\dot S(x,p)$.  Fix any $\delta < \kappa$ so that 
       \begin{enumerate}
       \item $I_1\cup I_2\subset \delta$,
       \item $\dot J_2$ is a nice $\Fn(\delta,2)$-name,
       \item for all $x\in I_2$ and $(x,p)\in \dot J_2$, $\dot S(x,p)$ is
        a nice $\Fn(\delta,2)$-name,
        \end{enumerate}
        and choose any $\mu\in U_\delta$ such that $\tilde q_\mu\in G$. 
        \bigskip
        
        We are ready for our final contradiction (to the assumption that
        $X$ has cardinality greater than $\kappa$). The contradiction
        will be that $z$ is not in the closure of the
        set $A_\mu = \{ \xi(h_\mu(\alpha),\mu) : \alpha \in I_1
         \ \mbox{and}\ h_\mu(\alpha)\in \val_G(\dot J_\mu)\}$, but 
         $z$ is in the closure of $ J_2$. We will show that 
         $J_2$ is contained in the closure of $A_\mu $. To show
         this, it suffices to prove that for each $x\in I_2$
         and $p\in G$ such that $(x,p)\in \dot J_2$, 
         the sequence $\{\bar\xi_\alpha : \alpha\in \val_G(\dot S(x,p))\} $ meets
          $A_\mu$ in an infinite set. Fix any $x\in I_2$
          and $p$ so that $(x,p)\in \dot J_2$. Let $H$ be any finite
          subset of $I_1$ and we will show that there is an
           $\alpha\in \val_G(\dot S(x,p))\setminus H$ such that $\bar\xi_\alpha
           \in A_\mu$.
           To do so, we will work in $V$ and have to use the genericity
           of $G$. Following \cite{Kunen}*{VII}, it suffices to show
           that if $r$ is any element of $G$, there is an $\alpha\in I_1\setminus H$
           and an extension $r_\alpha \in \Fn(\kappa,2)$ so that, with $\bar\xi_\alpha$
            $r_\alpha$ forces that $\alpha\in \dot S(x,p)$ and 
              $r_\alpha$ forces that $h_\mu(\alpha)\in \dot J_\mu$. 
This is the same as proving
            that the set of conditions of the form $r_\alpha$
            is dense below the condition $p$. This ensures that $G$ will include
            one such $r_\alpha$ and that $\bar\xi_\alpha$ will be the desired 
            member of $A_\mu$ with $\alpha\in \dot S(x,p)$.
            
          Now we use that $\{ q(h_\mu(\alpha),\mu) : \alpha \in I_1\setminus H\}$
          is a $\Delta$-system. 
            By possibly increasing the size of $H$ we can assume that 
            $\dom(r)\cup\dom(\tilde q_\mu)\cup
            \dom(p)$ is disjoint from  $\dom(q(h_\mu(\alpha), \mu))\setminus
            \dom(\tilde q_\mu)$ for all $\alpha\in I_1\setminus H$.
           Next we simply note that, since $\{ \bar\xi_\alpha : \alpha \in \val_G(
           \dot S(x,p))\}$
            converges to $x$, there an $\alpha\in I_1\setminus H$
           and a condition $\bar r_\alpha\in G$ so that $(\alpha,\bar r_\alpha)$ is an 
           element of the name $\dot S(x,p)$. Recalling that $p$
           forces that $\dot S(x,p)$ is a subset of $J_1$, it follows also
           that  $\bar r_\alpha \leq \bar q_\alpha$. 
           By the choice of $\delta$,
            $\bar r_\alpha\in \Fn(\delta,2)$  and  
            $\bar r_\alpha\leq q(h_\mu(\alpha),\mu)\restriction\mu$.
            Finally, it follows that with $r_\alpha=\bar r_\alpha\cup 
             q(h_\mu(\alpha),\mu)$ we have completed the proof.
\end{proof}

\section{Pseudoradial spaces in models of  $\cee \leq \aleph_2$}

In this section we prove that if $\cee$ is most $\aleph_2$, then there
are separable regular pseudoradial spaces of cardinality $2^{\cee}$. 
These examples are also 0-dimensional with a countable dense set
of isolated points. 
We first prove that CH implies the stronger result that there are 
such space that are compact.

\begin{theorem}[CH] There is  a compactification 
of $\omega$ that is pseudoradial and has cardinality
 $2^{\aleph_1}$.
 \end{theorem}
 
 \begin{proof} Let $X$ be an $\eta_1$-set of cardinality $\aleph_1$
 and let $\prec$ denote the corresponding (strict) linear ordering on $X$.  
 Recall that 
 an $\eta_1$-set has the property that if $A$ and $B$ are countable subsets
 of $X$ and that $a\prec b$ for all $a\in A$ and $b\in B$, then there
 is an $x\in X$ with $a\prec x \prec b$ for all $a\in A$ and $b\in B$.
 Next, let $\mathbb D$ denote the standard Dedekind completion of 
 $X$. Since $X$ is dense in $\mathbb D$ it follows that $\mathbb D$
 has weight $\aleph_1$ and  no points of countable character. It also
 follows then that 
 $\mathbb D$ has cardinality $2^{\aleph_1}$. Next, let $K$
 be the space obtained by doubling every point of $\mathbb D$
 that has   countable one-sided character. Again, using that $X$ is 
 dense, it follows that $K$ has weight $\aleph_1$. 
 Since $K$ is a linearly ordered space, it follows that $K$ is radial.
 We observe
 that $K$ has the property that every countably infinite set
 has subsequence converging to a point of countable character.

 Finally, it follows from  CH that there is a compactification
 $\gamma \omega$ that has remainder 
 $K$. This is the requisite example. It remains only   
 to prove that $\gamma\omega$ is sequentially compact.
Since $K$ is sequentially compact, it suffices to prove
that every infinite subset $A$ of $\omega$ has a converging
subsequence. If $A$ has only finitely many accumulation points
then this is clear. On the other hand, if $A$ has
infinitely many accumulation points, 
one of them  will be a point of countable character. 
                   \end{proof}

Now we turn to the proof in the
 case that  $\mathfrak c=\aleph_2$.
 Before constructing the example  we establish  an instructive
 constraint on the nature 
of possible examples.  This is based on the paper \cite{BaumWeese}
concerning partitioner algebras. A similar application in \cite{Dowlarge}
established that it was consistent that compact separable spaces of cardinality
greater than $\cee$ were not sequentially compact.  The proof is similar
to the proof of Theorem \ref{mble}.

\begin{theorem} It is consistent with $\mathfrak c=\aleph_2$
that\label{pseudo}
every pseudocompact pseudoradial regular space 
that contains a countable dense discrete space has cardinality
at most $\cee$.  Moreover,
if $V$ is a model of GCH and $\kappa>\aleph_1$ is a regular cardinal,
 then in the forcing extension by $\Fn(\kappa,2)$ every pseudocompact
 pseudoradial regular space with a countable dense discrete subset has
 cardinality at most $\kappa$.
\end{theorem}

\begin{proof} Clearly it suffices to prove the second statement of the 
Theorem. Suppose that GCH holds and that $\kappa>\aleph_1$ is a regular
cardinal. Let $G$ be a $\Fn(\kappa,2)$-generic filter.
We assume that $\theta>\kappa$ is a cardinal and that we
have a $\Fn(\kappa,2)$-name for a topology on $X=\theta$ 
such that, in $V[G]$,
$\omega$ is open, dense, and discrete. Similar to Theorem
\ref{mble}, let, for each $\alpha<\lambda$ and $p\in \Fn(\kappa,2)$,
 $\mathcal N_\alpha(p)$ be the set of all nice names of subsets of $\omega$
 that are forced by $p$ to have $\alpha$ in the interior of their closure. 
 Also let $\mathcal S_\alpha(p)$ be the set of all nice names of subsets
 of $\omega$ that are forced by $p$
 to simply have $\alpha$ in their closure. 
 To avoid confusion, we will let $1_P$ denote the maximal element
  $\emptyset$ of $\Fn(\kappa,2)$.

Since $\kappa$ is a regular cardinal, and we are assuming GCH, we 
may fix an elementary submodel $M$ of $H(\theta^+)$ of
cardinality $\kappa $ such that $\{\{\mathcal N_\alpha(1_P) ,
 \mathcal S_\alpha(1_P)\} :
\alpha\in \theta\}$  
is an element of $M$, and every subset of $M$ of cardinality less than
$\kappa$ 
is an element of $M$.
Let  $z$ be any element of $\theta\setminus M$. In this proof we will choose a
 $\kappa$-sequence $\{ y_\xi : \xi <\kappa\} \subset M\cap\theta $ 
 satisfying
 that, in $V[G]$, 
 the sequence $\{  y_\xi : \xi<\kappa\}$ converges to $z$. 
Fix  enumerations $\{ \dot S_\beta : \beta < \kappa \}$ of
 $\mathcal S_z(1_P)$
 and $\{ \dot W_\beta : \beta < \kappa\}$ of $\mathcal N_z(1_P)$.
 Fix any $\lambda<\kappa$. By assumption,
 the sequence $\{ \dot S_\beta : \beta <\lambda\}$
 and $\{\dot W_\beta : \beta < \lambda\}$ are elements of $M$. 
Since 
$$H(\theta^+)\models (\exists z\in\theta)( \{ \dot S_\beta : \beta < \lambda\}
\subset \mathcal S_z(1_P) \ \mbox{and}\ 
 \{ \dot W_\beta : \beta < \lambda\}
\subset \mathcal N_z(1_P) )$$ 
there is a $y_\lambda\in M\cap \theta$ satisfying
that $\{ \dot S_\beta : \beta < \lambda\}\subset \mathcal S_{y_\lambda}(1_P)$
and $ \{ \dot W_\beta : \beta < \lambda\}
\subset \mathcal N_{y_\lambda}(1_P)$. 

\medskip

It is easily verified that $\{ y_\lambda : \lambda < \kappa \}$ converges
to $z$ in $V[G]$ but this is not actually needed in the remainder of the proof. 
Rather, 
we will prove that $\omega$
is forced to have an infinite subset $S$ that contains
no converging sequence.   If $X$ is pseudoradial, $S$ would be closed
and, since $\omega$ is open in $X$, this would imply that $X$ is not
pseudocompact.
 
For limit $\lambda <\kappa$ of uncountable cofinality,
no cofinal subsequence
of
 $\{ y_\alpha : \alpha <\lambda\}$ converges to $z$
since $z$ is not an element of $M$. We translate this
into the forcing language.

\begin{claim}
For each limit\label{radial}
 $\lambda<\kappa$ of uncountable cofinality,
there is a value $\zeta_\lambda <\kappa$
and $\dot W_{\zeta_\lambda}\in \mathcal N_z(1_P)$ 
satisfying that, for all $\beta <\lambda $ and $p\in \Fn(\kappa,2)$,
there are $\beta < \alpha < \lambda$ and $p_\alpha < p$ such
that the name $\omega\setminus \dot W_{\zeta_\lambda}$ is an element
of $\mathcal N_{y_\alpha}(p_\alpha)$. 
\end{claim}

The last statement in Claim \ref{radial} is
simply asserting that $p_\alpha$ forces that $y_\alpha$
is not in the closure of $\dot W_{\zeta_\lambda}$. We also
note that the statement of Claim \ref{radial} asserts
that, for each $\beta<\lambda$,
the set of $p_\alpha$ as in the statement is a dense
subset of $\Fn(\kappa,2)$.  For each $\lambda$, choose
another  ordinal $\rho_\lambda$ so that $1_P$ forces
that the closure of $\dot W_{\rho_\lambda}\in \mathcal N_z(1_P)$ is 
contained in the interior of the closure of 
 $\dot W_{\zeta_\lambda}$.

Let $\Lambda$ denote the set of $\lambda<\kappa$ of uncountable
cofinality.
 For all $\lambda\in \Lambda$, let
$\dot W_{\zeta_\lambda}$ be the name
 $ \{ (n,p(\lambda,n,k)) :
 n\in B_\lambda, k\in\omega\}$
 and let $\dot W_{\rho_\lambda}$ be the name
  $\{ (n,q(\lambda,n,k)) : n\in D_\lambda, k\in\omega\}$.
  Note that $D_\lambda\subset B_\lambda$ and we may assume that,
  for all $n\in D_\lambda$ and $k\in \omega$, there is a $j\in\omega$
  so that $q(\lambda,n,k)<p(\lambda,n,j)$ (since
   $q(\lambda,n,k)\Vdash n\in \dot W_{\zeta_\lambda}$).
Let $I_\lambda$ 
denote the union of the set
       $\bigcup \{ \dom(q(\lambda,n,k)) : n\in B_\lambda, k\in
\omega\}$.
       For any $p\in\Fn(\kappa,2)$ and partial injection $\sigma :
       \kappa\mapsto \kappa$ 
       such that $\dom(p)\subset\dom(\sigma)$, we let $\widehat\sigma(p)$ 
       be the condition with domain $\sigma(\dom(p))$ and satisfying that
        $\widehat{\sigma}(p)(\sigma(\alpha)) = p(\alpha)$ for all
       $\alpha\in \dom(p)$. 
       
Fix a cub $C\subset\kappa$ satisfying that
for all $\mu\in C$,
$\{\zeta_\lambda\}\cup 
I_\lambda\subset \mu$ for all $\lambda \in \Lambda\cap \mu$.
By the usual pressing down lemma,
       and some straightforward enumeration techniques as used in the proof
       of Theorem \ref{mble}, there  is a stationary set $E\subset
       \Lambda\cap C$
       satisfying that, for all $\lambda,\mu \in E$ with $\lambda<\mu$,
  there is an order preserving isomorphism $\sigma_{\lambda,\mu}$
           from $ I_\lambda$ to $ I_\mu$ so that: 
           \begin{enumerate}
       \item $B_\lambda = B_\mu$ and
 $I_\lambda\cap\lambda = I_\mu\cap \mu$,
       \item $\sigma_{\lambda,\mu}(I_\lambda\setminus \lambda) = I_\mu\setminus
 \mu$,

        \item for all $n\in B_\lambda$ and $k\in\omega $, $
  p(\mu,n,k) = \widehat{\sigma}_{\lambda,\mu}(p(\lambda,n,k))$,
  \item for all $n\in D_\lambda$ and $k\in \omega$, 
   $ q(\mu,n,k) = \widehat{\sigma}_{\lambda,\mu}(q(\lambda,n,k))$ .
          \end{enumerate}
           Let $\mu_0$ be any element of $E$.
           Next, using that $E$ is stationary  it follows
           that we can find  $\mu_0<\mu_1\in E$
so that
           every $\Fn(\mu_1,2)$-name in $\mathcal S_{z}(1_P)$
           is in the list $\{\dot S_\beta : \beta<\mu_1\}$.
           Additionally we can
ensure that, for every nice $\Fn(\mu_0,2)$-name,
 $\dot S$,
           for a subset of $\omega$, and every $p\in \Fn(\kappa,2)$
           such that $\dot S\in \mathcal S_z(p)$,
           then $\dot S\in \mathcal S_z(p\restriction\mu_1)$.
           This latter condition is a simple consequence of
           the fact that $\Fn(\kappa,2)$ is ccc.

Next, extend $\{\mu_0,\mu_1\}$ to 
 any strictly increasing sequence $\{ \mu_\ell :
\ell\in\omega\}\subset E_1$. 
Our next step is to prove that the family
$\{ \dot W_{\zeta_{\mu_\ell}} : \ell\in \omega\}$ is forced
to be an independent family, or rather that
 $\{ \{ \dot W_{\rho_{\mu_\ell}}, \omega\setminus \dot W_{\zeta_{\mu_\ell}}\}: 
 \ell \in\omega\}$ is forced to be an independent family of pairs.
We will use the fact
       $\Fn(\kappa,2)$ is isomorphic to $ \Fn(\mu_0,2)\times
       \Fn(\kappa\setminus\mu_0,2)$ 
       and so when  we pass to the extension $V[G_{\mu_0}]$ (where
       $G_{\mu_0} = G\cap  
       \Fn(\mu_0,2)$), we can  use the fact that $V[G]$ is equal to 
       $V[G_{\mu_0}][G\cap \Fn(\kappa\setminus \mu_0,2)]$.
       We are trying to minimize our appeal to advanced forcing
       techniques but we  need a minor one here.

       \begin{claim}
         For each\label{reflect}
 $S\in [\omega]^{\aleph_0}\cap V[G_{\mu_0}]$,
         if $V[G]\models z\in \mbox{cl}_X(S)$, then
         there is an $\beta<\mu_1$ such that
          $\dot S_{\beta}\in \mathcal S_z(1_P)$ is a
         $\Fn(\mu_1,2)$-name and
         $\val_{G}(\dot S_\beta)=S$.
       \end{claim}

       \bgroup

       \def\proofname{Proof of Claim:\/}

       \begin{proof} By definition, if $S\in V[G_{\mu_0}]\cap
         [\omega]^{\aleph_0}$ there is a $\Fn(\mu_0,2)$-name
         $\dot S$ such that $S = \val_{G_{\mu_0}}(\dot S)$.
         Now assume that, in $V[G]$, that $z$ is in the closure
         of $\dot S$. By the forcing theorem, there is a
         condition $p\in G$ that forces $z$ is
         in the closure of $\dot S$. Of course
         this means that $\dot S\in \mathcal S_z(p)$.
         By the assumption on $\mu_1$, we can assume that
         $p\in \Fn(\mu_1,2)$. It is a routine exercise
         to prove that there is a nice $\Fn(\mu_1,2)$-name
         $\dot T$ such that $p\Vdash \dot T = \dot S$
         and,  all $q\in\Fn(\mu_1,2)$ that are incomparable
         with $p$ force that $\dot T = \omega$.
         It follows now, from the assumption on $\mu_1$,
         that $\dot T$ is in the list
          $\{ \dot S_\alpha : \alpha < \mu_1\}$. 
       \end{proof}

       \egroup

       Let $B=\{n\in B_{\mu_0} : (\exists k) p(\mu_0,n,k)\restriction
       \mu_0 \in G_{\mu_0}\}$ and
       $D=\{n\in D_{\mu_0} : (\exists k) q(\mu_0,n,k)\restriction
       \mu_0 \in G_{\mu_0}\}$.
       Let $I$ denote the countable set $I_{\mu_0}\setminus \mu_0$
       and let, for each $\ell<\omega$, $\dot U_\ell$ denote
       the canonical $\Fn(I_{\mu_\ell}\setminus \mu_\ell, 2)$-name
       in $V[G_{\mu_0}]$ for the set $\dot W_{\zeta_{\mu_\ell}}$
       and let $\dot V_\ell$ denote the canonical 
       $\Fn(I_{\mu_\ell}\setminus \mu_\ell, 2)$-name
       in $V[G_{\mu_0}]$ for the set $\dot W_{\rho_{\mu_\ell}}$.
       More precisely, for each $n\in B$ and $k\in \omega$,
       $$(n,p(\mu_\ell,n,k)\restriction I) \in \dot U_0
\ \ \mbox{if and only if}\ \ p(\mu_\ell,n,k)\restriction\mu_0\in
       G_{\mu_0}.$$
 For all $\ell>0$,
       $\dot U_\ell = \{(n,\widehat\sigma_{\mu_0,\mu_\ell}(p)) :
 (n,p)\in \dot U_0\}$.  For convenience, let $\sigma_{\mu_0,\mu_0}$ denote
 the identy function on $I$ and
  recall that the sets
  $\{ I_{\mu_\ell}\setminus \mu_\ell = \sigma_{\mu_0,\mu_\ell}(I)
   : \ell\in\omega\}$ are pairwise disjoint.
The definition of $\dot V_\ell$ is defined similarly.

 For each $p\in \Fn(I,2)$, let 
$T_p = \{ n\in B : p\Vdash
 n\in \dot U_0\}$
 and let 
$V_p = \{ n\in D : (\exists q < p ) ~q\Vdash n\in \dot V_0\}$.
 We note that $T_p$ and $V_p$ are
 elements of
 $V[G_{\mu_0}]$.

\begin{claim}
For each  $p\in \Fn(I,2)$,
 $z$ is not in the closure, in $V[G]$,
of $T_p$.
Similarly, $z$ is in the interior of the closure
of $V_p$.
\end{claim}

       \bgroup

       \def\proofname{Proof of Claim:\/}

       \begin{proof}
 Assume that the first statement fails  for some $p\in \Fn(I,2)$.
 By Claim \ref{reflect}, there is an $\beta < \mu_1$ such that
  $T_p = \val_G(\dot S_\beta)$. 
 The family $\{ \sigma_{\mu_0,\mu_\ell}(p) : 1\leq \ell\in\omega\}$
 is a $\Delta$-system (with empty root) and so there is
 an $\ell>0$ such that $\sigma_{\mu_0,\mu_\ell}(p)\in G$.
 By Claim \ref{radial}, there is an $\alpha$ and a condition
 $p_\alpha<\sigma_{\mu_0,\mu_\ell}(p)$ such that
 $p_\alpha\restriction \mu_0\in G_{\mu_0}$,
  $ \beta  < \alpha <\mu_\ell$ and $p_\alpha\Vdash \omega\setminus
 \dot W_\ell\in \mathcal N_{y_\alpha}(p_\alpha)$.
 This is equivalent to the assertion that $p_\alpha$
 forces that $y_{\alpha}$ is not in the closure of $\dot W_\ell$.
 On
 the other hand, we have that $\dot S_\beta\in \mathcal
 S_{y_\alpha}(1_p)$.  This implies that $1_P$
 forces that $y_\beta$ is in
 the closure of $T_p$ and yet, in $V[G_{\mu_0}]$,
 $p_\alpha$ forces $T_p 
 \subset \dot U_\ell$ is disjoint from a neighborhood of
 $y_\alpha$.

 The second statement is proven similarly by simply noting
 that $\sigma_{\mu_0,\mu_\ell}(p)$ forces that $\dot V_\ell$
 is a subset of $V_p$.
\end{proof}

       \egroup

 Now we prove that $\{ \{\dot V_\ell, \omega\setminus
 \dot U_\ell\} : \ell\in\omega\}$ is forced
 to be an independent family.  Fix any pair of disjoint finite
 sets $L_0,L_1$ of $\omega$. Let $q\in \Fn(\kappa\setminus \mu_0,2)$
 be any condition. We produce a condition $\tilde q<q$ and an integer
 $n$ so that $\tilde q$ forces that $n\in \dot V_\ell$ for all
 $\ell\in L_0$ and $\tilde q$ forces that $n\notin \dot U_\ell$
 for all $\ell\in L_1$. 
 For each $\ell\in L_0\cup L_1$, let $p_\ell\in \Fn(I,2)$
 be chosen so that $\sigma_{\mu_0,\mu_\ell}(p_\ell) = q\restriction
 I_{\mu_\ell}\setminus \mu_\ell$. Let $T = \bigcup\{ T_{p_\ell} : \ell\in
 L_1\}$
 and $U  = \bigcap \{V_{p_\ell} : \ell\in L_0\}$. 
 Choose  any integer $n\in U\setminus T$.
 For each $\ell\in L_0$,  $n\in U_{p_0}$
 and so
we may choose $q_\ell<p_\ell$
 in $\Fn(I,2)$, 
 such that $q_\ell\Vdash n\in \dot V_0$.
 For each $\ell\in L_1$, since $n\in T_{p_\ell}$, 
$p_\ell\not\Vdash n\in \dot U_0$, 
and so we may choose $q_\ell<p_\ell$
 in $\Fn(I,2)$ such that $q_\ell\Vdash n\notin \dot U_0$.
 It now follows that $\tilde q = q\cup \bigcup\{ \sigma_{\mu_0,\mu_\ell}(q_\ell) :
 \ell\in L_0\cup L_1\}$ is an element of $\Fn(\kappa\setminus
 \mu_0,2)$
 and that $\tilde q\Vdash n\in \dot V_\ell$ for all $\ell\in L_0$
 and $\tilde q\Vdash n\notin \dot U_\ell$ for all $\ell
\in L_1$ as required.

\bigskip

We reassure the reader that we are nearly at the end of the proof. Fix
any $\lambda < \kappa$ such that $\mu_\ell < \lambda$ for all
$\ell\in\omega$ and let $\dot L_\lambda$ be the canonical
subset of $\omega$ given by $n\in \dot L_\lambda$ if and only
if $\bigcup G (\lambda+n) = 1$. That is $\dot L_\lambda$
is the name $\{ (n, \{\langle \lambda+n,1\rangle\}
) : n\in\omega\}$. We have proven that the family
$\{ \dot V_\ell : \ell\in \dot L_\lambda \}\cup
\{ \omega\setminus \dot U_\ell : \ell \in \omega\setminus \dot
L_\lambda\}$ has the finite intersection property. 
We can define  $\dot S$ to be the sequence $\{ \dot n_\ell :
\ell\in \omega\}$ where $\dot n_\ell$ is the minimum element 
above $\dot n_{\ell-1}$ (for $\ell>0$)
of the set $$\bigcap\{ \dot V_j : j\leq \ell \ \mbox{and}\ j\in   \dot L_\lambda\}\setminus
 \bigcup \{ \dot U_j : j \leq \ell \ \mbox{and}\ j\notin \dot L_\lambda\}.$$
Each $\dot n_\ell$ exists was proven in the previous paragraph.
Now suppose that $\dot J$ is a nice $\Fn(\kappa\setminus\mu_0,2)$-name 
of an infinite subset of $\omega$. Choose any $\mu\in E$ satisfying that
 $\lambda+\omega < \mu$ and $\dot J$ is a $\Fn(\mu\setminus \mu_0,2)$-name.
 We complete the proof by showing that each of $\dot W_{\rho_\mu}$ and
  $\omega\setminus W_{\zeta_\mu}$ hit the sequence
   $\{ \dot n_\ell : \ell \in \dot J\}$ in an infinite set. Since $\dot W_{\rho_\mu}$
   and $\omega\setminus W_{\zeta_\mu}$ have disjoint closures, and $\dot J$
   was arbitrary, this shows that $\dot S$ does not have any limit points. 
   The proof is a density argument.   Analogous to the definitions of
   $\dot V_\ell,\dot W_\ell$ from $\dot W_{\rho_{\mu_\ell}}, \dot W_{\zeta_{\mu_\ell}}$,
   let us define $\dot V_\mu = \{ (n,\widehat\sigma_{\mu_0,\mu}(p)) : (n,p)\in \dot V_0\}$
   and $\dot U_\mu = \{ (n,\widehat\sigma_{\mu_0,\mu}(p)) : (n,p)\in \dot U_0\}$. 
It follows that $\dot V_\mu$ is the $\Fn(\kappa\setminus \mu_0,2)$-name for
 $\val_{G}(\dot W_{\rho_\mu})$ and $\dot U_\mu$ is the 
 $\Fn(\kappa\setminus \mu_0,2)$-name for
 $\val_{G}(\dot W_{\zeta_\mu})$.
   Let $p$ be an arbitrary element of 
$\Fn(\kappa\setminus \mu_0,2)$ and let $m$ be any integer. We will
   produce $q<p$ and a pair $m<\ell_1<\ell_2$ so that
    $q$ forces that $\{\ell_1,\ell_2\}\subset \dot J$, 
     $\dot n_{\ell_1}\in \dot V_{\mu}$ and $\dot n_{\ell_2}\notin \dot 
      U_\mu$.  By possibly increasing $m$, we can assume
      that $\dom(p)\cap [\lambda,\lambda+\omega)$ is
      contained in $[\lambda,\lambda+m]$. Similarly,
       we can assume that $\dom(p)\cap I_{\mu_\ell}\setminus \mu_\ell$
       is empty for all $\ell>m$. 
Choose $\bar p\in \Fn(I,2)$ so that $\sigma_{\mu_0,\mu}(\bar p)$
       is equal to $p\restriction I_\mu\setminus \mu$. Let $q_1
       \in\Fn(\mu\setminus\mu_0,2)$
       be any extension of
        $(p\restriction\mu) \cup \sigma_{\mu_0,\mu_{m+1}}(\bar p)$ 
       such that $q_1(\lambda{+}m{+}1)=1$, and
        there is an $\ell_1>m$ and an $n_{\ell_1}\in\omega$ 
       such that $q_1\Vdash \ell_1\in \dot J$ and $q_1\Vdash 
        \dot n_{\ell_1} = n_{\ell_1}$. Notice that $q_1\Vdash n_{\ell_1}
        \in \dot V_{m+1}$. Now let $\bar{p}_1\in \Fn(I,2)$ be chosen
        so that $\sigma_{\mu_0,\mu_{m+1}}(\bar{p}_1) = 
           q_1\restriction I_{\mu_{m+1}}\setminus \mu_{m+1}$
           and note that $\bar{p}_1<\bar p$. Let us note
           that $\sigma_{\mu_0,\mu}(\bar {p}_1)$ forces that $n_{\ell_1}
           \in \dot V_\mu$. Choose next a sufficiently large $m_2 > \ell_1$
           so that $\dom(q_1)\cap [\lambda,\lambda+\omega)
           \subset [\lambda, \lambda+m_2)$ and 
            $\dom(q_1) \cap I_{\mu_\ell}\setminus \mu_\ell$ is empty
            for all $\ell >m_2$. Choose $q_2\in \Fn(\mu\setminus \mu_0,2)$
            so that $q_2 < \sigma_{\mu_0,\mu_{m_2}}(\bar {p}_1)$,
             $q_2(\lambda+m_2)  = 0$, and, such that there are
             an $\ell_2>m_2$ and $n_{\ell_2}\in\omega$
             such that $q_2\Vdash \ell_2\in \dot J$ and $q_2\Vdash 
             \dot n_{\ell_2 } = n_{\ell_2}$. 
             Let $\bar{p}_2\in \Fn(I,2)$ such that $\sigma_{\mu_0,\mu_{m_2}}(\bar{p}_2)
              = q_2\restriction I_{\mu_{m_2}}\setminus \mu_{m_2}$.
Since $q_2\Vdash m_2\notin \dot L_\lambda$, 
we have that $q_2\Vdash n_{\ell_2}\in \omega\setminus \dot U_{{m_2}}$. 
Therefore $\bar{p}_2\Vdash n_{\ell_2}\notin \dot U_{0}$, 
 which implies that 
 $\sigma_{\mu_0,\mu}(\bar{p}_2)<\sigma_{\mu_0,\mu}(\bar{p}_1)$
  forces that $n_{\ell_2}\notin
    \dot U_{\mu}$ and $n_{\ell_1}\in \dot V_{\mu}$.                         
              The required condition $q$ is $q_2 \cup \sigma_{\mu_0,\mu}(\bar{p}_2)$.
We have that $\sigma_{\mu_0,\mu}(\bar {p}_2)$ forces that $n_{\ell_1}\in \dot V_\mu$
and that $n_{\ell_2}\notin \dot U_\mu$.        
\end{proof}

We do not know if we had to assume that $\omega$ was a discrete
subset in Theorem \ref{pseudo} so was raise this question.

\begin{question} Is it consistent that every separable pseudocompact
pseudoradial regular space has cardinality at most $\cee$?
\end{question}

Now we prove what may be the main result of the paper.

\begin{theorem} If $\mathfrak c\leq\aleph_2$, then there
  is a separable 0-dimensional pseudoradial space of cardinality
  greater than $\mathfrak c$.
\end{theorem}

\begin{proof}
As we learned in our previous
results the space will necessarily contain a large independent family
of clopen sets, and in fact the family must have special properties
that will support the existence of many converging sequences.
We will use the ideas 
introduced by P. Simon
\cite{Petr} 
to facilitate the construction of such families
by using a countable set with special structure.
 For each $n\in\omega$, let $D_n = \left(2^n\right)^{2^n}$
Our countable dense set will be $D =
\bigcup \{ D_n : n\in \omega\}$.  

Our construction will be to define a Boolean subalgebra $\mathcal B$
of $\mathcal P(D)$
and to then choose a subspace of the space of ultrafilters, i.e. Stone space,
 of $\mathcal B$.   As usual, the finite subsets of $D$ will be members
 of $\mathcal B$ and 
 we identify the fixed ultrafilters of $\mathcal B$ with the points of $D$.

 We will define a subtree   of $2^{<\omega_2}$
 that has cardinality $\mathfrak c$ and has more than
 $\mathfrak c$ many cofinal branches. If $2^{\aleph_1}=\aleph_2$
 then this tree will necessarily have height $\omega_2$.
 We first present the proof in the case that
 $2^{\aleph_1}=\aleph_2$ and indicate at the end of
 the proof the minor modification needed for the case
 that  $2^{\aleph_1}>\aleph_2$.

We will let 
 $T$ denote the chosen subtree of $2^{\leq\omega_2}$ (including
 its cofinal branches)
 and use
  the elements of $T$ to enumerate the ultrafilters of $\mathcal B$
  that are chosen to be members of our space $X$. This approach
  ensures that our space is zero-dimensional and Hausdorff.
Thus $X$
  will, technically have as
  base set $D\cup \{\mathcal F_t : t\in T\}$ but when proving
  that $X$ is pseudoradial it will be convenient to 
 identify  $t\in T$ with the point $\mathcal F_t$ in $S(\mathcal B)$.  
  The structure of $T$ will be very
 helpful in understanding the convergence 
  structure of $X$. In particular, for many,
   but not all,  elements $t$ of $T$ that lie
  on limit levels of uncountable cofinality,
   the set of predecessors of $t$ will be a well-ordered sequence that
   converges to $t$. This gives some of the insight into the proof
   that $X$ is pseudoradial.  As mentioned in the first
   paragraph,
the structure of $D$ is chosen to allow
   us to more strategically define the sequences from $D$ that
   converge to points     of $X$. 
 
  For $t\in 2^{\leq\omega_2}$, let $o(t)$ denote
the domain or level of $t$. 
Let $S^2_0$ and $S^2_1$ denote the stationary
subsets of $\omega_2$ consisting of the cofinality $\omega$ and,
respectively,
$\omega_1$ limits.
Let $(S^2_1)'$ denote the set of all limits
of $S^2_1$ and let $S_0 = S^2_0\cap (S^2_1)'$ and
$S_1 = S^2_1\cap (S^2_1)'$. Let $T_0$ be the subtree
 of $2^{\leq\omega_2}$
where
$t\in T_0$ if and only if $t^{-1}(1) $ is a subset of
 $S_1$ (i.e. we are only branching at levels in $S_1$). 
 It should be clear that $T_0\cap 2^{\omega_2}$ has cardinality
  $2^{\aleph_2}$.  We will need to add some more nodes to our
  tree because we want to
have a large set of auxillary  successors to associate with each
$t\in T_0$ where $o(t)\in S_0$. Therefore $T$ will consist of
$T_0$ together with
all $t^\frown \sigma$ (concatentation) where
$t\in T_0$, $\sigma\in 2^{\leq\omega}$, and $o(t)\in S_0$.
For each $t\in T\setminus T_0$, let $\delta_t\in S_0$
and $\sigma_t\in 2^{\leq\omega}$ denote the values
such that $t = t\restriction\delta_t{}\frown \sigma_t$.
For $s,t\in T$, let $s\wedge t$ be the maximum element
of $T$ satisfying that $s\wedge t\leq s$ and
 $s\wedge t\leq s$. 

\medskip

Now we choose our special independent subfamily of
$\mathcal P(D)$ and a special family of converging sequences.
It will make the proof more readable
 to use the set $T\cap 2^{<\omega_2}$ to
enumerate the family.
 Let $\{ r_t : t\in T_0\cap 2^{<\omega_2}\}$
be any one-to-one enumeration of $2^\omega\setminus \{\vec0\}$
(where $\vec0$ is the constant zero function).
$$\mbox{For each}~ t\in T_0\cap 2^{<\omega_2},\ \ 
 A_t = \bigcup_n \{ d\in D_n
: d(r_t\restriction n) = r_t\restriction n\}~.$$ 
Let $\mathcal L$ denote the family of finite non-empty
chains of $T_0\cap 2^{<\omega_2}$. 
For each  $L\in\mathcal L$,
let
$B_L $ be the set of $d\in \bigcap\{ A_t : t\in L\}$ satisfying
that $d(\sigma)=\vec 0\restriction n$ 
providing $d\in D_n$ and $\sigma\in 2^n\setminus
\{r_t\restriction n:   t\in L\}$
(i.e. $d\in B_L\cap D_n$ if $d$ is the identity
on $\{ r_t\restriction n : t\in L\}$ and otherwise takes on
value  $\vec 0\restriction n$).

\begin{claim}  For each $L\in \mathcal L$, $B_L$ is an
infinite subset of $\bigcap \{ A_t : t\in L\}$ and\label{indf} 
for  $L\neq L'\in \mathcal L$, 
 $B_L\cap B_{L'}$ is finite.
\end{claim}

\bgroup

\def\proofname{Proof of Claim:\/}

\begin{proof}
By symmetry we may assume that $L\setminus L'\neq\emptyset$ and
let $t_L\in L\setminus L'$. Choose $n_0$ large enough so that the elements
of $\{ r_t\restriction n_0 : t \in L\cup L'\}$  are pairwise distinct and 
not equal to $\vec{0}\restriction {n_0}$. 
For all $n>n_0$ and $d\in B_{L'}$, $d(r_{t_L}\restriction n) =
\vec{0}\restriction n$  
and, for all $d\in B_{L}$, $d(r_{t_L}\restriction n)=
r_{t_L}\restriction 
n\neq \vec{0}\restriction n$. 
\end{proof}

It follows from Claim \ref{indf} that the family $\{ A_t :
t\in T_0\cap 2^{<\omega_2}\}$ 
is an independent family. 
Let $\mathcal I$ be the ideal of
subsets of $D$ that are almost disjoint from every such
$B_L$, i.e. $\mathcal I$ is often denoted as
$\left(\{ B_L : L \in \mathcal L\}\right)^\perp$.
\medskip

Let $\mathcal B$ denote the Boolean
subalgebra of $\mathcal P(D)$ generated by the family
$$ [D]^{<\aleph_0}\cup 
\{ A_t : t\in T_0\cap 2^{<\omega_2}\} \cup \{ B_L : 
L\in \mathcal L\}
\cup \mathcal I
$$
and we will choose, for each $t\in T$, a (free) ultrafilter
 $\mathcal F_t$
 on
$\mathcal B$. 

For each $t\in T$, we will first make
an assignment $ H_t$
consisting of a non-empty set of
 nodes in $T$ that are less or equal to $t$,
and
this assignment will determine the ultrafilter $\mathcal F_t$
as described below.

\medskip

We begin with the assignment $\langle H_t : t\in T\rangle$.

\begin{enumerate}
 \item 
   For $t\in T_0$ with $o(t)\in S_1\cup \{\omega_2\}$,  $H_t =
   \{ t\restriction\alpha : \alpha\in o(t)\}$.
\item 
 For each $t\in T_0$ such that $o(t)=\mu+\omega_1$
 for some  $\mu\in S_1$,
ensure that
$\{ H_{t\restriction\xi} : \mu< \xi <\mu+\omega_1\}$ is a one-to-one
  listing
  of all the finite subsets of 
$\{ t\restriction \alpha : \alpha <\mu+\omega_1\}$ that have at
  least one element strictly extending  $t\restriction\mu$
   (and satisfy the requirement that $t\restriction\xi\in
  H_{t\restriction\alpha} $ implies $\xi\leq\alpha$).
\item 
 For each $t\in T_0$ such that $o(t)=\mu+\omega_1$
 for some  $\mu\in \{0\}\cup \overline{S^2_1}\setminus S_1$,
ensure that
  $\{ H_{t\restriction\xi} : \mu\leq \xi < \mu+\omega_1\}$ is a one-to-one
  listing
  of all the finite subsets of 
$\{ t\restriction \alpha : \alpha <\mu+\omega_1\}$ that have at
  least one element  above $t\restriction\mu$
   (and satisfy the requirement that $t\restriction\xi\in
  H_{t\restriction\alpha} $ implies $\xi\leq\alpha$).
\item 
  For each $t\in T_0$ with $o(t)\in S_0$, ensure that\\
$\{ H_{t^\frown \sigma} : \sigma\in 2^{\leq\omega}\ \mbox{and}\
  t^\frown \sigma\notin T_0\}$ is a one-to-one listing of
  all  countable cofinal subsets of 
$\{ t\restriction \alpha : \alpha< o(t)\}$.
 Note that, in the previous item,  $H_{t^\frown
    \sigma}$ has already been 
  defined as a finite set if 
    $ t^\frown \sigma\in T_0$.
    \end{enumerate}

\begin{claim}
  If $t\neq s$ and $t,s\in T$, then $H_t\neq H_s$.
\end{claim}

 \begin{proof}
   It is clear that $H_t\neq H_s$ if either of them is
   uncountable. Similarly, if they are countably infinite,
   then it follows from clause (4) that $\delta_t=\delta_s$,
   and thus that $H_t\neq H_s$ since the assignment
   in (4) was chosen to be one-to-one.  Finally
   we suppose that $H_t$ and $H_s$ are finite. 
We may
   choose minimal $\mu_t$ so that $\mu_t\leq o(t)\leq \mu_t+\omega_1$.
   Similarly choose minimal $\mu_s$ so that
   $\mu_s\leq o(s)\leq \mu_s+\omega_1$.
   It follows  by the minimality, that $\mu_t$ and $\mu_s$
   are limit points of $S^2_1$. 
   It follows easily from clauses (2) and (3)
   that if $H_t=H_s$, then $\mu_t=\mu_s$.

   If $\mu_t\in S_1$  clause (2)
   ensures that if
   $t\restriction\mu_t+1 =s\restriction \mu_t+1$
   then $H_s,H_t$  are distint elements
   of $\{ H_{\bar t\restriction \xi} : \mu_t <\xi < \mu_t+\omega_1\}$
     for some $\bar t\in 2^{\mu_t+\omega_1}$.
     Otherwise
     $t\restriction\mu_t+1 \neq s\restriction \mu_t+1$
     and clause (2) ensures that
      $H_t$ and $H_s$ contain incomparable elements.
     Finally, if $\mu_t \notin S_1$, then 
$s$ and $t$ are distinct but comparable
     and
again we have
     that $H_t$ and $H_s$ are
     distinct elements of
     the list
 $\{ H_{\bar t\restriction \xi} : \mu_t \leq \xi < \mu_t+\omega_1\}$
     for some $\bar t\in 2^{\mu_t+\omega_1}$.     
   \end{proof}

\medskip

\noindent Now we make the simple assignment of the
family $\langle \mathcal F_t : t\in T\rangle$:

\begin{enumerate}
  \setcounter{enumi}{4}
\item If $H_t$ is  finite, then define
$\mathcal F_t$ to be the unique
  free ultrafilter of $\mathcal B$ with $B_{H_t}\in\mathcal F_t$.
\item If $H_t$ is infinite, then $\mathcal F_t$ is the ultrafilter
  satisfying
  \begin{enumerate}
  \item $A_s\in \mathcal F_t$ if and only if $s\in H_t$,
  \item $B_L\notin\mathcal F_t$ for all $L\in\mathcal L$,
  \item $I\notin\mathcal F_t$ for all $I\in \mathcal I$,
  \item $[D]^{<\aleph_0}\cap \mathcal F_t$ is empty.
  \end{enumerate}
\end{enumerate}

 It should be clear that the subspace $D\cup X_1$ is simply
 equal to the traditional Mrowka-Isbell type space constructed
 from the almost disjoint family $\{ B_L : L\in\mathcal L\}$.
 Let us also note that for each $I\in\mathcal I$,
 $I$ is closed in $X$, and this implies this next claim.
 
  \begin{claim} If $Y\subset D$ and\label{11}
 $\overline{Y}\cap X_1$ is empty,
    then $Y$ is closed.
\end{claim}

  \begin{proof} If no point of $X_1$ is in the closure of
    $Y$, then $Y\cap B_L$ is finite for all $L\in\mathcal L$.
    Thus $Y\in \mathcal I$.
  \end{proof}

It will be convenient to let $K_t = \{A_s : s\in
 (T_0\cap2^{<\omega_2})\setminus H_t\}$, i.e.
$K_t= \{ A_s : A_s\notin \mathcal F_t\}$. 
For disjoint
finite subsets
 $H, K $ of $T_0\cap 2^{<\omega_2}$,
 let 
$[H;K] $
 denote the clopen subset of $X$
 corresponding to the closure of this element  of $\mathcal B$:
 $$[H;K] = \bigcap\{A_t: t\in H\} \setminus\left(
 \bigcup\{A_s : s\in K\}\cup \bigcup\{B_L : \emptyset\neq L\subset H\}
\right)~.$$
 For any $t\in X_2\cup X_3$,
the family $\{ [H;K] : H\in [H_t]^{<\aleph_0}, K\in
[K_t]^{<\aleph_0}\}$ is easily seen to be a local filter base in
the subspace $X\setminus D$.
 Let $X_0 = D$, $X_1 = \{ t\in T_0 : |H_t|<\aleph_0 \}$,
 $X_2 = \{ t\in T : |H_t|=\aleph_0\}$
 and let
 $X_3=\{t\in T: |H_t|>\aleph_0\}$.

 \medskip

 \begin{claim} $X_3$ is a closed radial\label{x3closed}
 subspace of $X$.
 \end{claim}

 \begin{proof} We first show that $X_3$ is closed. Clearly
   $X_0\cup X_1$ is open so consider any $t\in X_2$.
   Since $H_t$ is a countable cofinal subset of the uncountable
   ordinal $\delta_t$, there are $\beta<\alpha<\delta_t$
   such that $t\restriction \beta\in K_t$ and $t\restriction \alpha
   \in H_t$. Note that
   $[\{t\restriction\alpha\};\{t\restriction\beta\}]$
   is disjoint from $X_3$ since $H_s$ is a downward closed subset
   of $T_0$.

   Now consider any $Y\subset X_3 $ and assume that
   $t\in X_3$ is a limit point of $Y$. We will assume
   that $o(t)<\omega_2$ and leave the case when $o(t)=\omega_2$
   to the reader. 
Choose any strictly
   increasing cofinal 
   sequence $\{ \alpha_\xi : \xi \in\omega_1\}\subset o(t)$.
   Note that $t\in K_t$. For each
   $\xi<\omega_1$, choose $y_\xi\in Y\cap
   [ \{t\restriction \alpha_\xi\}; \{ t\}]$.
   Consider any finite $H\subset H_t$ and finite $K\subset K_t$.
   Let $\tilde K = \{ s\in K_t : t\not\leq s\}$ and choose
   $\xi <\omega_1$ so that $t\restriction\alpha_\xi\not \leq s$ for
   all $s\in \tilde K$. It follows easily
   that $y_\eta\in[H;K]$ for all $\xi<\eta <\omega_1$
   and thereby proving that $\langle y_\xi : \xi\in\omega_1\rangle$
   converges to $t$.
 \end{proof}

 We continue the proof
 that $X$ is pseudoradial. Let us note that it follows
 from Claim \ref{11} that it suffices to proof
 that $X\setminus D$ is pseudoradial. For the remainder
 of the proof we will say that an elementary submodel $M$
is
suitable to mean that $M\prec H(\aleph_3)$ and
that $\{ T_0,\mathcal L,
\{A_t : t\in T_0\cap 2^{<\omega_2}\},
\{H_t : t\in T\}, \{B_L : L\in \mathcal L\}\}$ is an
element of $M$. Here is one of the key properties
 of the space.

 \begin{claim} If $Y\subset X\setminus D$ and\label{x2}
 $t\in X_2\cup X_3$
   is a limit point of $Y$, then for any countable
suitable  elementary submodel $M$ 
such that $Y,t\in M$,
there is a converging sequence $\{ t_n : n\in\omega\}\subset Y\cap M$ 
such that $t$ is the limit if $t\in X_2$ and, if $t\in X_3$,
the limit is $t_M\in X_2$ where
\begin{enumerate}
\item $\delta_{t_M} = \sup(M\cap o(t))$ and
  $t_M = t\restriction\delta_{t_M}{}^\frown \sigma_{t_M}$,
  \item  $H_{t_M} = \{ t\restriction \alpha : \alpha\in M\cap
    o(t)\}$.
\end{enumerate}
 \end{claim}

 \begin{proof} 
   Let $\bar t$ equal to $t$ if $t\in X_2$ and let
   $\bar t= t_M$ if $t\in X_3$. In either case, we have
   that $H_{\bar t}\subset M$. If $t\in X_3$, then observe
   that $t\in K_t$ and
$K_t\cap \{ t\restriction\alpha : \alpha\in M\cap o(t)\}$ is
   empty. If $t\in X_2$, then, by elementarity,
 $K_t \cap
   \{ t\restriction\alpha : \alpha\in M\cap o(t)\}$ is non-empty.
   Define $K_0\in M$ to be $\{t\}$ if $t\in X_3$
   and to be any finite subset of
$K_t \cap
   \{ t\restriction\alpha : \alpha\in M\cap o(t)\}$
   if $t\in X_2$. Choose $H_0$ to be any finite subset
   of $H_{\bar t}$ so that, if $t\in X_2$,
   then there are $\alpha<\beta$ so that
   $t\restriction \alpha\in K_0$
   and $t\restriction \beta \in H_0$.
 Since $[H_0;K_0]\in M$, there
   is a $y_0\in M\cap Y\cap [H_0;K_0]$.

Choose any sequence  $\{ [H_n;K_n] : n\in \omega\}\subset M$
so that:
 \begin{enumerate}
\item  $\{ H_n : n\in \omega\}$ is a strictly increasing sequence of 
 finite sets whose union is $H_{\bar t}$,
\item  $\{ K_n : n\in  \omega\}$
  is a strictly increasing sequence of finite sets whose
union equals        $(K_{\bar t}\setminus H_t)\cap M$,
 \end{enumerate}

We show that we can choose $y_n\in Y\cap M\cap [H_n;K_n]$ for all
$n$. If $t=\bar t$, then $t  \in [H_n;K_n]$ and
so $Y\cap M\cap [H_n;K_n]$ is not empty by elementarity and the
assumption that $t$ is a limit point of $Y$.
If $t\neq\bar t$, then $H_n\subset H_{\bar t}\subset H_t$
and $ K_n \subset K_{\bar t}\setminus H_t$.
Therefore, we again  have that $t\in [H_n;K_n]$ 
and so, by elementarity,
there is a $y_n\in Y\cap M\cap [H_n;K_n]$.

Now we show that the sequence $\{y_n : n\in\omega\}$ converges to
$\bar t$.  Consider any finite $H\subset H_{\bar t}$ and finite
 $K\subset K_{\bar t}$. Choose $\bar\beta<\delta_{\bar t}$ large enough
 so that $o(s\wedge (t\restriction\delta_{\bar t}) )<\bar \beta$
 for all $s\in K$ such that $t\restriction \delta_{\bar t}\not\leq s$.
 We may assume that $t\restriction \bar\beta\in H_{\bar t}$.
Choose $n$ so that $H\cup \{t\restriction\bar\beta\}\subset H_n$, 
 and $(K\setminus H_t)\cap M$
 is a subset of $K_n$. We show that $y_m\in [H;K]$ for
 each $m>n$. Since $H\subset H_m$ and $y_m\in [H_m;K_m]$
 it is clear that $H\subset H_{y_m}$ and $K_m\subset K_{y_m}$.
 Also, $H_{y_m}\setminus H_m$ is not empty (in the case
 that $H_{y_m}$ is finite). To show that  $y_m\in [H;K]$,
 we must simply show that
 $H_{y_m}\cap K$ is empty.

 In the case that $t=\bar t$, then $K_{\bar t}\setminus H_t =
 K_{\bar t}$, and so $H_{y_m}\cap K \subset H_{y_m}\cap K_m
 =\emptyset$. So now we assume that $\bar t \neq t\in X_3$
 and recall that $t\in K_0$. 
 In this case, $K_{\bar t}\cap M$ is also disjoint from
  $H_t$, hence $K\cap M\subset K_n$.
If $y_m\notin X_3$,
 then $H_{y_m}\subset M$ and so $H_{y_m}\cap K \subset H_{y_m}\cap
 K\cap M \subset H_{y_m}\cap K_m =\emptyset$.
Now suppose that $y_m\in X_3$ and assume that $s\in H_{y_m}\cap K$.
Note that  we have that 
 $t\restriction \bar\beta \in H_{y_m}$ and so $t\restriction\delta_{\bar
  t}\leq s \leq y_m$. 
Also, 
 $\{ t\restriction \alpha : \alpha
<\delta_{\bar t}\} $ is a subset of $H_{y_m}$. 
Since $t\in K_0$ implies that $o(t\wedge y_m)\in
M\cap o(t) =\delta_{\bar t}$,
this 
contradicts  that $t\restriction
 \delta_{\bar t}\leq s$.
\end{proof}

 \begin{claim} If $Y\subset X\setminus D$ and $t\in X_3$  is a limit point
   of\label{x3} $Y$, then there is a  sequence
   $\{ y_\alpha : \alpha\in\omega_1\}$ converging to a point in $X_3$
   where each $y_\alpha$ is in  the sequential closure of $Y$.
 \end{claim}

 Before we prove this Claim, we observe that this completes
 the proof that $X$ is pseudoradial. If $Y$ is a
 subset of $X\setminus D$ then, by Claim \ref{x2} every
 limit point in $X_2$ is in the sequential closure, and
 by Claim \ref{x3}, the points  of $X_3$
in the radial closure
of $Y$ are dense in $\overline{Y}\cap X_3$. Since $X_3$
is radial (Claim \ref{x3closed}), this implies that the radial closure
of $Y$ is  closed.

\begin{proof}
 Fix any increasing sequence $\{ M_\alpha : \alpha\in\omega_1\}$ of
suitable countable elementary submodels  satisfying
that $Y,t$ are elements of $M_0$. 
There is an ordinal $\theta\in S_1$ such that
$\theta = \bigcup\{M_\alpha\cap \omega_2 : \alpha\in\omega_1\}$
 (see \cite{Kunen}).
 If $o(t)<\omega_2$, let $\lambda = o(t)\in S_1$. 
If $o(t)=\omega_2$,
  let $\lambda = \theta $.

   For each $\alpha\in\omega_1$,
   let $\delta_\alpha = \sup(M_\alpha\cap \lambda)$.
   Following Claim \ref{x2}, let
   $y_\alpha  = t\restriction \delta_\alpha{}^\frown \sigma_\alpha$
   where $\sigma_\alpha$ is chosen so that 
 $H_{y_\alpha} =\{ t\restriction\gamma : \gamma\in  M_\alpha\cap \lambda\}$.
   It follows from Claim \ref{x2}, that $y_\alpha$ is in the
   sequential closure of $Y$.

   Now we simply check that $\{y_\alpha : \alpha\in\omega_1\}$ converges
   to $t\restriction\lambda$. Fix any basic open
   set $[H;K]$ of $t\restriction\lambda$.
Choose
any $\beta < \omega_1$ so that $H\in M_\beta$.
Observe that $H\subset H_{y_\alpha}$ for all
$\beta<\alpha\in\omega_1$.
Similarly, $H_{y_\alpha}\subset H_{t\restriction\delta_\alpha}\subset
H_{t\restriction\lambda}$ for all $\alpha\in\omega_1$.
Therefore $K\cap H_{y_\alpha}=\emptyset$, and
so $y_\alpha\in[H;K]$,
  for all $\beta<\alpha\in\omega_1$.
This completes the proof of the Claim.
\end{proof}
\egroup

This completes the proof of the Theorem in the case
that $2^{\aleph_1}=\aleph_2$. The only modification that
is needed for the case $2^{\aleph_1}>\aleph_2$ is
to pass to the subtree $T_0\cap 2^{\leq\theta}$
where $\theta\in S_1$ is the minimum level of $T_0$
which has cardinality greater than $\aleph_2$.
This ensures that $T_0\cap 2^{<\theta}$ has
cardinality $\aleph_2$.  Establish
the new enumeration  $\{ r_t : t\in T_0\cap 2^{<\theta}\}$
of $ 2^{\omega}\setminus\{0\}$ and  the proof proceeds exactly as above.
\end{proof}

\begin{question} Is it consistent that there is no separable
  pseudoradial space of cardinality $2^{\mathfrak c}$?
\end{question}

\begin{question} Do separable regular pseudoradial spaces of
cardinality greater than $\mathfrak c$ exist? 
\end{question}

 We also do not know if large separable Hausdorff pseudoradial spaces
 of cardinality greater than $\mathfrak c$ exist.


\begin{bibdiv}

\def\cprime{$'$} 

\begin{biblist}

\bib{Arhangelskii}{article}{
   author={Arhangel\cprime ski\u{\i}, A. V.},
   title={The structure and classification of topological spaces and
   cardinal invariants},
   language={Russian},
   journal={Uspekhi Mat. Nauk},
   volume={33},
   date={1978},
   number={6(204)},
   pages={29--84, 272},
   issn={0042-1316},
   review={\MR{526012}},
}

\bib{ArhBella93}{article}{
   author={Arhangel\cprime skii, A. V.},
   author={Bella, A.},
   title={On the cardinality of a pseudo-radial compact space},
   language={English, with Italian summary},
   journal={Boll. Un. Mat. Ital. A (7)},
   volume={7},
   date={1993},
   number={2},
   pages={237--241},
   review={\MR{1234074}},
}

\bib{Balogh}{article}{
   author={Balogh, Zolt\'{a}n T.},
   title={On compact Hausdorff spaces of countable tightness},
   journal={Proc. Amer. Math. Soc.},
   volume={105},
   date={1989},
   number={3},
   pages={755--764},
   issn={0002-9939},
   review={\MR{930252}},
   doi={10.2307/2046929},
}

\bib{BaumWeese}{article}{
   author={Baumgartner, James E.},
   author={Weese, Martin},
   title={Partition algebras for almost-disjoint families},
   journal={Trans. Amer. Math. Soc.},
   volume={274},
   date={1982},
   number={2},
   pages={619--630},
   issn={0002-9947},
   review={\MR{675070}},
   doi={10.2307/1999123},
}

\bib{Bella86}{article}{
   author={Bella, Angelo},
   title={Free sequences in pseudoradial spaces},
   journal={Comment. Math. Univ. Carolin.},
   volume={27},
   date={1986},
   number={1},
   pages={163--170},
   issn={0010-2628},
   review={\MR{843428}},
}



\bib{vD}{article}{
   author={van Douwen, Eric K.},
   title={The integers and topology},
   conference={
      title={Handbook of set-theoretic topology},
   },
   book={
      publisher={North-Holland, Amsterdam},
   },
   date={1984},
   pages={111--167},
   review={\MR{776622}},
}

\bib{Dowlarge}{article}{
   author={Dow, Alan},
   title={Large compact separable spaces may all contain $\beta {\bf N}$},
   journal={Proc. Amer. Math. Soc.},
   volume={109},
   date={1990},
   number={1},
   pages={275--279},
   issn={0002-9939},
   review={\MR{1010799}},
   doi={10.2307/2048389},
}

\bib{PrahaIII}{article}{
   author={Dow, Alan},
   title={Set-theoretic update on topology},
   conference={
      title={Recent progress in general topology. III},
   },
   book={
      publisher={Atlantis Press, Paris},
   },
   date={2014},
   pages={329--357},
   review={\MR{3205487}},
   doi={10.2991/978-94-6239-024-9\_7},
}
\bib{picharacter}{article}{
   author={Dow, Alan},
   title={Countable $\pi$-character, countable compactness and PFA},
   journal={Topology Appl.},
   volume={239},
   date={2018},
   pages={25--34},
   issn={0166-8641},
   review={\MR{3777320}},
   doi={10.1016/j.topol.2018.02.009},
}

\bib{DowSide}{article}{
   author={Dow, Alan},
   title={Generalized side-conditions and Moore-Mr\'{o}wka},
   journal={Topology Appl.},
   volume={197},
   date={2016},
   pages={75--101},
   issn={0166-8641},
   review={\MR{3426909}},
   doi={10.1016/j.topol.2015.10.016},
}
 		
\bib{DowFeng}{article}{
   author={Dow, Alan},
   author={Feng, Ziqin},
   title={Compact spaces with a P-base},
   journal={preprint},
   date={2020},
%   pages={226--238},
%   issn={0166-8641},
%   review={\MR{3414886}},
%   doi={10.1016/j.topol.2015.09.025},
}	

\bib{DowTodd}{article}{
   author={Dow, Alan},
   author={Eisworth, Todd},
   title={CH and the Moore-Mrowka problem},
   journal={Topology Appl.},
   volume={195},
   date={2015},
   pages={226--238},
   issn={0166-8641},
   review={\MR{3414886}},
   doi={10.1016/j.topol.2015.09.025},
}	

\bib{Todd}{article}{
   author={Eisworth, Todd},
   title={Countable compactness, hereditary $\pi$-character, and the
   continuum hypothesis},
   journal={Topology Appl.},
   volume={153},
   date={2006},
   number={18},
   pages={3572--3597},
   issn={0166-8641},
   review={\MR{2270606}},
   doi={10.1016/j.topol.2006.03.021},
}

\bib{Jech}{book}{
   author={Jech, Thomas},
   title={Set theory},
   series={Perspectives in Mathematical Logic},
   edition={2},
   publisher={Springer-Verlag, Berlin},
   date={1997},
   pages={xiv+634},
   isbn={3-540-63048-1},
   review={\MR{1492987}},
   doi={10.1007/978-3-662-22400-7},
}
	 	
\bib{JKL}{article}{
   author={Juh\'{a}sz, Istv\'{a}n},
   author={Koszmider, Piotr},
   author={Soukup, Lajos},
   title={A first countable, initially $\omega_1$-compact but non-compact
   space},
   journal={Topology Appl.},
   volume={156},
   date={2009},
   number={10},
   pages={1863--1879},
   issn={0166-8641},
   review={\MR{2519221}},
   doi={10.1016/j.topol.2009.04.004},
}

\bib{JSpseudo}{article}{
   author={Juh\'{a}sz, I.},
   author={Szentmikl\'{o}ssy, Z.},
   title={Sequential compactness versus pseudo-radiality in compact spaces},
   journal={Topology Appl.},
   volume={50},
   date={1993},
   number={1},
   pages={47--53},
   issn={0166-8641},
   review={\MR{1217695}},
   doi={10.1016/0166-8641(93)90071-K},
}


		\bib{Kunen}{book}{
   author={Kunen, Kenneth},
   title={Set theory},
   series={Studies in Logic and the Foundations of Mathematics},
   volume={102},
   note={An introduction to independence proofs},
   publisher={North-Holland Publishing Co., Amsterdam-New York},
   date={1980},
   pages={xvi+313},
   isbn={0-444-85401-0},
   review={\MR{597342}},
}

\bib{Justin}{article}{
   author={Moore, Justin Tatch},
   title={Open colorings, the continuum and the second uncountable cardinal},
   journal={Proc. Amer. Math. Soc.},
   volume={130},
   date={2002},
   number={9},
   pages={2753--2759},
   issn={0002-9939},
   review={\MR{1900882}},
   doi={10.1090/S0002-9939-02-06376-1},
}
		


\bib{Boris1}{article}{
   author={\v{S}apirovski\u{\i}, B. \`E.},
   title={$\pi $-character and $\pi $-weight in bicompacta},
   language={Russian},
   journal={Dokl. Akad. Nauk SSSR},
   volume={223},
   date={1975},
   number={4},
   pages={799--802},
   issn={0002-3264},
   review={\MR{0410632}},
}
		
\bib{Boris}{article}{
   author={\v{S}apirovski\u{\i}, B. \`E.},
   title={Mappings on Tihonov cubes},
   language={Russian},
   note={International Topology Conference (Moscow State Univ., Moscow,
   1979)},
   journal={Uspekhi Mat. Nauk},
   volume={35},
   date={1980},
   number={3(213)},
   pages={122--130},
   issn={0042-1316},
   review={\MR{580628}},
}
 
\bib{SapPseudo}{article}{
   author={Shapirovski\u{\i}, Boris},
   title={The equivalence of sequential compactness and pseudoradialness in
   the class of compact $T_2$-spaces, assuming CH},
   conference={
      title={Papers on general topology and applications},
      address={Madison, WI},
      date={1991},
   },
   book={
      series={Ann. New York Acad. Sci.},
      volume={704},
      publisher={New York Acad. Sci., New York},
   },
   date={1993},
   pages={322--327},
   review={\MR{1277868}},
   doi={10.1111/j.1749-6632.1993.tb52534.x},
}

\bib{Petr}{article}{
   author={Simon, P.},
   title={Applications of independent linked families},
   conference={
      title={Topology, theory and applications},
      address={Eger},
      date={1983},
   },
   book={
      series={Colloq. Math. Soc. J\'{a}nos Bolyai},
      volume={41},
      publisher={North-Holland, Amsterdam},
   },
   date={1985},
   pages={561--580},
   review={\MR{863940}},
}

\bib{Boban}{article}{
   author={Veli\v{c}kovi\'{c}, Boban},
   title={Forcing axioms and stationary sets},
   journal={Adv. Math.},
   volume={94},
   date={1992},
   number={2},
   pages={256--284},
   issn={0001-8708},
   review={\MR{1174395}},
   doi={10.1016/0001-8708(92)90038-M},
}


\end{biblist}
\end{bibdiv}
\end{document}